\documentclass[12pt, letter paper]{amsart}
\usepackage{graphicx, amsmath, amssymb, amsthm, amsfonts, fullpage, latexsym,enumerate, xcolor, commath, tikz, tikz-cd, verbatim, hyperref,standalone, mathtools, booktabs}
\usepackage[capitalize]{cleveref}
\usepackage[shortlabels]{enumitem}
\usepackage[backend=biber,style=alphabetic]{biblatex}
\usepackage{thmtools}
\title{Graded Ehrhart Theory of Unimodular Zonotopes}

\author{Colin Crowley and Ethan Partida}

\theoremstyle{definition}
\newtheorem{defn}{Definition}[section]
\newtheorem{remark}[defn]{Remark}
\newtheorem{notation}[defn]{Notation}
\newtheorem{example}[defn]{Example}
\theoremstyle{plain}
\newtheorem{lemma}[defn]{Lemma}
\newtheorem{prop}[defn]{Proposition}
\newtheorem{corollary}[defn]{Corollary}

\newtheorem{thrm}[defn]{Theorem}

\newtheorem{question}[defn]{Question}
\newtheorem{problem}[defn]{Problem}

\newcommand{\R}{\mathbb{R}} 
\newcommand{\Z}{\mathbb{Z}}
\newcommand{\Q}{\mathbb{Q}} 
\newcommand{\C}{\mathbb{C}}
\newcommand{\N}{\mathbb{N}}

\renewcommand{\P}{\mathbb{P}}
\newcommand{\FF}{\mathbb{F}}
\renewcommand{\O}{\mathcal{O}}

\renewcommand{\H}{\mathcal{H}}
\DeclareRobustCommand{\qbinom}{\genfrac{\lbrack}{\rbrack}{0pt}{}}

\DeclareMathOperator\pr{\mathrm{pr}}

\DeclareMathOperator\Sym{\mathrm{Sym}}

\DeclareMathOperator\gr{\mathrm{gr}}
\DeclareMathOperator\Hilb{\mathrm{Hilb}}

\DeclareMathOperator\rowsupp{\mathrm{rowsupp}}
\DeclareMathOperator\colsupp{\mathrm{colsupp}}
\DeclareMathOperator\rowspace{\mathrm{rowspace}}
\DeclareMathOperator\ehr{\mathrm{ehr}}

\DeclareRobustCommand{\red}[1]{ {\begingroup\color{red}{#1}\endgroup} }

\definecolor{NordBlack}{HTML}{3B4252}   
\definecolor{NordWhite}{HTML}{E5E9F0}      
\definecolor{NordCyan}{HTML}{8FBCBB}  
\definecolor{NordBlue}{HTML}{81A1C1}

\begin{document}

\begin{abstract}
Graded Ehrhart theory is a new $q$-analogue of Ehrhart theory based on the orbit harmonics method. We study the graded Ehrhart theory of unimodular zonotopes from a matroid-theoretic perspective. Generalizing a result of Stanley (1991), we prove that the graded lattice point count of a unimodular zonotope is a q-evaluation of its Tutte polynomial. We conclude that the graded Ehrhart series of a unimodular zonotope is rational and obeys graded Ehrhart--Macdonald reciprocity. In an algebraic direction, we prove that the harmonic algebra of a unimodular zonotope is a coordinate ring of its associated arrangement Schubert variety. Using the geometry of arrangement Schubert varieties, we prove that the harmonic algebra of a unimodular zonotope is finitely generated and Cohen--Macaulay. We also give an explicit presentation of the harmonic algebra of a unimodular zonotope in terms of generators and relations. We conclude by classifying which unimodular zonotopes have Gorenstein harmonic algebras. Our work answers, in the special case of unimodular zonotopes, two conjectures of Reiner and Rhoades (2024).\end{abstract}

\maketitle

\section{Introduction}
In this paper, we study the combinatorial, algebraic and geometric aspects of the graded Ehrhart theory of unimodular zonotopes. Graded Ehrhart theory is a new $q$-analogue of Ehrhart theory introduced by Reiner and Rhoades \cite{RR24}. For every lattice polytope $P$ and positive integer $m$, there is a canonical pair of polynomials $i_P(m;q)$ and $\widetilde{i_P}(m;q)$ in $\Z[q]$ that have non-negative coefficients summing to the number of lattice points and interior lattice points in $mP$, respectively. These polynomials are constructed via the \emph{orbit harmonics} method.  The orbit harmonics method is a combinatorial degeneration that turns a finite set of points $\mathcal{Z}\subseteq k^n$ into a graded $k$-algebra $\text{Orb}(\mathcal{Z})$ whose total dimension is equal to $\vert \mathcal{Z}\vert$. When the finite set of points $\mathcal{Z}$ is combinatorially interesting, the dimensions of the graded pieces of $\text{Orb}(\mathcal{Z})$ often have an interesting combinatorial interpretation. For example, if $\mathcal{Z}$ is the set of vertices of the permutahedron, then the ring $\text{Orb}(\mathcal{Z})$ is the coinvariant ring of the symmetric group \cite[Section 4.1]{HRS}. The polynomials $i_P(m;q)$ and $\widetilde{i_P}(m;q)$ record the dimensions of the graded pieces of $\text{Orb}(mP\cap \Z^d)$ and $\text{Orb}(\text{int}(mP)\cap \Z^d)$, respectively.

A compelling feature of Ehrhart theory and, indeed, our present work, is its broad connections to combinatorics, commutative algebra and algebraic geometry. Classical Ehrhart theory is closely related to the enumerative combinatorics of integer-valued polynomials, the commutative algebra of semigroup rings, and the algebraic geometry of toric varieties. In our study of the graded Ehrhart theory of unimodular zonotopes, we draw from, and prove new facts about: the enumerative combinatorics of quantum integer-valued polynomials \cite{HH}, the commutative algebra of zonotopal and harmonic algebras \cite{HR,RR24}, and the algebraic geometry of arrangement Schubert varieties \cite{AB,liImagesRationalMaps2018}.  

The classical Ehrhart theory of unimodular zonotopes is studied in the works of Stanley, Beck--Jochemko--McCullough, and Backman--Baker--Yuen \cite{stanleyzonotope, BJM19, BBY19}. They show that, among many other things,  the Ehrhart theory of a unimodular zonotope $Z$ is dictated by its associated matroid $M$. Our results show that $M$ also dictates the graded Ehrhart theory of $Z$. 

We now summarize our results.  As our work draws on many different 
mathematical areas, we have separated our results into two categories: 
enumerative and algebraic. We delay our background until 
\Cref{sec:background} in order to highlight our results with a large 
worked example in \Cref{sec:example}. We expand upon the enumerative and 
algebraic aspects of our work in Sections \ref{sec:combinatorics} and 
\ref{sec:algebra}, respectively. We close by classifying which 
unimodular zonotopes have Gorenstein harmonic algebras in 
\Cref{Gorenstein-section}. \Cref{sec:questions} contains future questions.

\subsection{Enumerative Results}

Let $Z = A\cdot [0,1]^n \subseteq \R^d$ be a full-dimensional unimodular zonotope and $M$ be the matroid of $A$. See \Cref{subsec:matroids} and \Cref{subsec:zonotopes} for background on matroids and zonotopes, respectively.

Our first enumerative theorem expresses the graded lattice point counts $i_Z(m;q)$ and $\widetilde{i_Z}(m;q)$ of a unimodular zonotope as $q$-evaluations of the Tutte polynomial. This theorem, upon setting $q=1$, specializes to a result of Stanley \cite{stanleyzonotope} that expresses the ordinary lattice point counts $i_Z(m;1)$ and $\widetilde{i_Z}(m;1)$ as evaluations of the Tutte polynomial.

\begin{prop}\label{thrm:q-ehrhart} 
Let $m\geq 1$ be an integer, $Z=A\cdot [0,1]^n\subseteq \R^d$ be a unimodular zonotope and $M$ be the matroid of $A$. We have that,
  \begin{equation*}
    i_Z(m;q) =  q^{(n-d)m}[m]_q^d 
    T_M(\frac{[m+1]_q}{[m]_q},q^{-m})\,\,\,\text{and}\,\,\,
    \widetilde{i}_Z(m;q) = q^{(n-d)m}[m]_q^d 
    T_M(\frac{[m-1]_q}{[m]_q},q^{-m}),
\end{equation*}
where $T_M(x,y)$ is the Tutte polynomial of $M$ and $[k]_q := \frac{1-q^k}{1-q}$ is the $k$th $q$-integer.
\end{prop}
Our proof of \Cref{thrm:q-ehrhart} relies on the theory of zonotopal algebras. We give a short introduction to zonotopal algebras in \Cref{subsec:zonotopes} and prove \Cref{thrm:q-ehrhart} in \Cref{subsec:ehrhart-poly}. 

We use \Cref{thrm:q-ehrhart} to prove that the graded lattice point counts $i_Z(m;q)$ and $\widetilde{i_Z}(m;q)$ are evaluations of \emph{quantum integer-valued polynomials}. A polynomial $f(t)\in \mathbb{Q}(q)[t]$ is quantum integer-valued if $f([m]_q)\in \mathbb{Z}[q,q^{-1}]$ for all $q$-integers $[m]_q$. Quantum integer-valued polynomials are a natural $q$-analogue of integer-valued polynomials and were systematically studied in the work of Harman and Hopkins \cite{HH}. See \Cref{subsec:quantum} for a short introduction to quantum integer-valued polynomials.

\begin{thrm}[\Cref{prop:ehrhart_poly}]\label{thrm:ehrhart_poly}
    Let $Z=A\cdot [0,1]^n\subseteq \R^d$ be a unimodular zonotope.  There is a quantum integer-valued polynomial $\ehr_Z(t;q)\in \Q(q)[t]$, which we call the graded Ehrhart polynomial, such that $\ehr_Z([m]_q;q) = i_Z(m;q)$. The polynomial $\ehr_Z(t;q)$ exhibits a $q$-analogue of Ehrhart--Macdonald reciprocity.
\end{thrm}

In \Cref{prop:q-reciprocity}, we prove that the generating function of any quantum integer-valued polynomial is a rational function in $\mathbb{Q}(q,t)$ which obeys a formal quantum reciprocity law. This result is analogous to the fact that integer-valued polynomials have rational generating functions that obey a formal reciprocity law \cite{stanleyRecip, crt} and may be of independent interest.

From \Cref{thrm:ehrhart_poly} and \Cref{prop:q-reciprocity}, we conclude the following theorem about the graded Ehrhart series of unimodular zonotopes.

\begin{thrm}[\Cref{thrm:ehrhart_series}]\label{thrm:ehrhart_series_intro} 
Conjecture 1.1 of \cite{RR24} holds for unimodular zonotopes. Namely, the graded Ehrhart series
        \[E_Z(t,q) = \sum_{m\geq 0} i_Z(m;q) t^m \quad \text{and} \quad \widetilde{E_Z}(t,q) = \sum_{m\geq1} \widetilde{i_Z}(m;q)t^m  \] 
        are rational functions in $\Q(q,t)$ with a particularly nice 
        form. The series $E_Z(t,q)$ and $\widetilde{E_Z}(t,q)$ exhibit a 
        $q$-analogue of Ehrhart--Macdonald reciprocity. Namely, 
        \[q^d \widetilde{E_Z}(t,q) = (-1)^{d+1}E_Z(t^{-1}, q^{-1}).\]
\end{thrm}

\subsection{Algebraic Results}
For any lattice polytope $P$, there is a bigraded $\C$-algebra 
$\mathcal{H}_P$ and a homogeneous ideal $\widetilde{\mathcal{H}_P}$ of 
$\mathcal{H}_P$ whose bigraded Hilbert series are the graded Ehrhart 
series $E_P(t,q)$ and $\widetilde{E_P}(t,q)$, respectively. The algebra 
$\mathcal{H}_P$ is called the \emph{harmonic algebra} of $P$ and the 
ideal $\widetilde{\mathcal{H}_P}$ is called the \emph{interior ideal} of 
$P$\footnote{The harmonic algebra was originally defined in \cite{RR24} 
as an $\R$-algebra $\H_P^\R$ and the interior ideal as an ideal 
$\widetilde{\H_P^\R} \subseteq \H_P^\R$. In order to uniformly state our 
theorems, we work with the complexifications $\H_P \coloneqq \H_P^\R 
\otimes_\R \C$ and $\widetilde{\H_P}\coloneqq 
\widetilde{\H_P^\R}\otimes_\R \C$. All of our algebraic theorems continue to hold 
for the case of real coefficients by \cite[Lemma 1.1]{stanleyHilb}. }. 
See \Cref{sec:gradedEhrhart} for the definitions.

Our main algebraic theorem gives an algebro-geometric interpretation of 
the harmonic algebra of a unimodular zonotope. If $Z=A\cdot 
[0,1]^n\subseteq \R^d$ is a unimodular zonotope and $L=\rowspace(A)$, the \emph{arrangement 
Schubert variety} $Y_L$ of $L$ is a complex $d$-dimensional subvariety 
of $(\P^1)^n$, defined as the closure of the rowspace of $A$. See 
\Cref{sec:matroid-schubert} for more details on arrangement Schubert varieties.

The arrangement Schubert variety $Y_L$ carries a natural 
$\mathbb{C}^*$-action which restricts to the inverse scaling action on the rowspace of $A$. This action induces a bigrading on the homogeneous coordinate ring of any $\mathbb{C}^*$-equivariant embedding of $Y_L$ into projective space. It is with this bigrading that we state the following theorem.

\begin{thrm}\label{thrm:geometry}
    Let $Z=A\cdot [0,1]^n\subseteq \R^d$ be a unimodular zonotope. The 
    harmonic algebra $\H_Z$ is isomorphic to the bigraded homogeneous coordinate ring of $Y_L$ under the Segre embedding $Y_L\subseteq \P^{2^n-1}$. 
\end{thrm}

 This geometric perspective allows us to verify all conjectures that \cite{RRT23} make about $\H_Z$.

\begin{thrm}\label{thrm:CM-intro}
Conjecture 5.5 of \cite{RR24} holds for unimodular zonotopes. 
Namely, the harmonic algebra $\H_Z$ is a finitely generated, Cohen--Macaulay $\C$-algebra. After a shift by $q$-degree $d$, the canonical module $\Omega \H_Z$ is isomorphic to the interior ideal $\widetilde{\H_Z}$. 
\end{thrm}

Our proof of \Cref{thrm:CM-intro} ultimately rests on \cite[Theorem 1]{Brion}, and the fact that $Y_L \subseteq (\P^1)^n$ has multidegree which is multiplicity-free. In recent work, Cavey \cite{Cavey} has disproven \cite[Conjecture 5.5]{RR24} for general lattice polytopes. Using a toric geometric interpretation of harmonic algebras, he shows that $\mathcal{H}_P$ can fail to be finitely generated for $P$ a lattice triangle. In light of Cavey's result, it is quite surprising that the harmonic algebras of unimodular zonotopes are so well behaved.

We also give an explicit presentation of the vanishing ideal of $Y_L\subseteq \P^{2^n-1}$ and, in turn, a presentation of the harmonic algebra of $Z$.

\begin{thrm}[\Cref{prop:harmonic-pres}]\label{thrm:pres-intro}
For any linear subspace $L \subseteq \C^n$, the homogeneous coordinate ring of $Y_L \subseteq \P^{2^n-1}$ has an explicit combinatorial presentation
    \[R_L \simeq \frac{\C[z_S : S \subseteq \{1,\ldots, n\}]}{I_L^\text{SE}} \]
    where $\deg(z_S)= (1,|S|)$ and $I_L^\text{SE}$ is a homogeneous ideal constructed from the circuits of $M$.
\end{thrm}

Combining \Cref{thrm:geometry} and \Cref{thrm:pres-intro}, we obtain a presentation of the harmonic algebra of a unimodular zonotope.

\begin{corollary}
\label{cor:harmonic_pres}
The harmonic algebra of a unimodular zonotope $Z$ has an explicit combinatorial presentation, as given in \Cref{thrm:pres-intro}.
\end{corollary}

The homogeneous coordinate ring of $Y_L \subseteq 
            (\P^1)^n$ has an explicit presentation given in \cite{AB}, 
            so it is possible to compute 
            defining equations of the Segre embedding $Y_L \subseteq 
            \P^{2^n-1}$ in specific examples using Gr\"obner bases. However, such Gr\"obner bases appear to lack any combinatorial structure. For example, they often fail to have a squarefree initial ideal. Our generating set of $I_L^{\text{SE}}$ is not a Gr\"obner basis but is closely dictated by the combinatorial structure of $M$. This lets us prove that our generating set indeed generates $I_L^{\text{SE}}$ by an inductive deletion/contraction argument.

Finally, inspired by the well-known characterization of which polytopes give rise 
to Gorenstein semigroup algebras \cite[Theorem 
3.8.4]{coxToricVarieties2011}, we characterize which unimodular 
zonotopes have Gorenstein harmonic algebras.

\begin{thrm}\label{thrm:gor}
    Let $Z = A \cdot [0,1]^n$ be a unimodular zonotope, and let $M$ be 
    the matroid of $A$. The harmonic algebra $\mathcal{H}_Z$ is Gorenstein if and only if either $M$ is the Boolean matroid or every connected component of $M$ is a circuit.

\end{thrm}

While the matroids that show up in \Cref{thrm:gor} are quite simple, the graded Ehrhart theory appears to be nontrivial, as we see in the example of \Cref{sec:example}. Indeed, when one begins asking refined questions about the cube, many of the answers are highly nontrivial. See the survey paper of Zong \cite{zong05}. If every connected component of $M$ is a circuit, then $Z$ is a product of oriented matroid circuit polytopes corresponding to the classical type $A$ root system \cite{EMW24}. The facial structure and Ehrhart theory of such polytopes are studied in detail in the work of Escobar and McWhirter \cite{EMW24}.

We prove in \Cref{prop:palin} that if $Z$ has a Gorenstein harmonic algebra, then the numerator of $E_Z(t,q)$ displays a remarkable quantum-palindromic symmetry. This symmetry is evocative of Hibi's palindromic theorem which states that reflexive polytopes have symmetric $h^*$ vectors \cite{hibi}.

\subsection*{Acknowledgments} 
We thank Melody Chan, Steven Creech, Galen Dorpalen-Barry, Carly Klivans, Matt Larson, Ruizhen Liu, Nicholas Proudfoot, Vic Reiner, Brendon Rhoades and Benjamin Schr\"oter for helpful conversations and comments. We also thank Benjamin Braun for suggesting that we investigate which unimodular zonotopes have Gorenstein harmonic algebras. Finally, we thank the anonymous FPSAC 2026 referees for their helpful comments on an extended abstract of this project. Their comments influenced the organization of this manuscript. Colin Crowley was partially supported by NSF grant DMS-2039316 and the Simons Foundation. Ethan Partida was partially supported by NSF Grant DMS-2053288, a U.S. Department of Education GAANN award, and the Simons Foundation SFI-MPS-SDF-00015018.

\section{Illustration of results}\label{sec:example}
We now illustrate our results by working out the graded Ehrhart theory of a hexagon in the plane. Let $Z\subseteq \R^2$ be the zonotope defined by the matrix $A$ in the following figure.

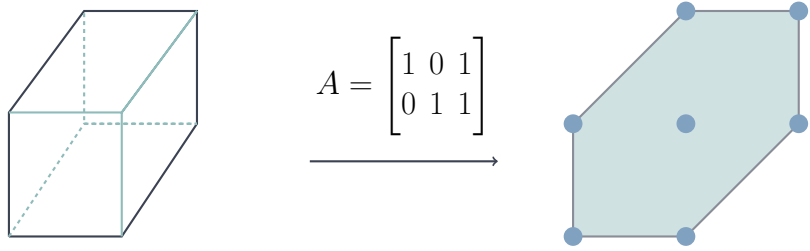
\begin{figure}[!h]
\centering
\scalebox{0.5}{
 \begin{tikzpicture}
    \draw[draw=NordBlack, ultra thick] (0,0) -- (3,0) -- (5,3) -- (5,6) -- (2,6) -- (0,3.3) -- (0,0);

    \draw[draw=NordCyan, ultra thick] (0,3.3) -- (3,3.3) -- (5,6) -- (3,3.3) -- (3,0);
    
    \draw[draw=NordCyan, dashed, ultra thick] (0,0) -- (2,3) -- (5,3) -- (2,3) -- (2,6);

    \draw[draw=NordBlack, ultra thick, -to] (8,2) -- (13,2);

    \node (iz) at (10.5,4) {\Huge$A =\begin{bmatrix}
    1 & 0 & 1\\
    0 & 1 & 1
    \end{bmatrix}$};
    
    \filldraw[draw=NordBlack, fill=NordCyan!60!NordWhite, opacity=0.6, ultra thick] (15,0)--(18,0)--(21,3)--(21,6)--(18,6)--(15,3)--(15,0);
    \fill[NordBlue] (15,0) circle (0.25cm);
    \fill[NordBlue] (18,0) circle (0.25cm);
    \fill[NordBlue] (15,3) circle (0.25cm);
    \fill[NordBlue] (18,3) circle (0.25cm);
    \fill[NordBlue] (21,3) circle (0.25cm);
    \fill[NordBlue] (18,6) circle (0.25cm);
    \fill[NordBlue] (21,6) circle (0.25cm);
    \end{tikzpicture}
}
\end{figure}
From the matrix $A$ defining $Z$, we can see that the matroid $M$ of $Z$ 
is the uniform matroid $U_{2,3}$ with bases $\{12,13,23\}$. We now 
calculate $i_Z(1;q)$ in two different ways. The polynomial $i_Z(1;q)$ is 
the Hilbert series of the ring $\text{Orb}(Z\cap \Z^2)$ 
(\Cref{defn:ehr}). In \Cref{thrm:zon_orbit}, we prove that  
$\text{Orb}(Z\cap \Z^2)$ is isomorphic to the external zonotopal algebra 
of $Z$ (\Cref{defn:zon}). 

The external zonotopal algebra of $Z$ is 
defined from the arrangement of hyperplanes normal to the columns of 
$A$, pictured below.
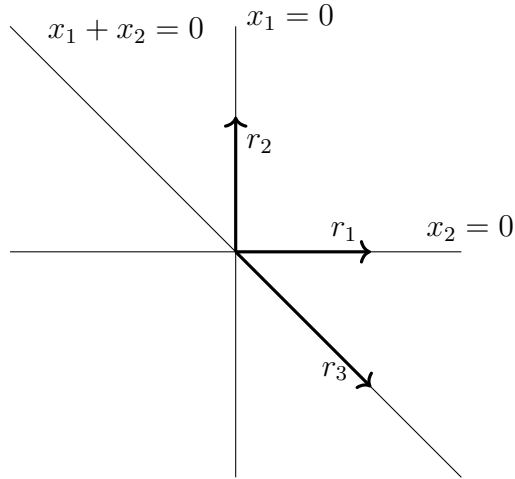
\begin{center}
\begin{tikzpicture}[scale=1.2]
  \draw[black] (0,-2.5) -- (0,2.5);
  \node[black, right] at (0,2.6) {$x_1=0$};
  \draw[black] (-2.5,0) -- (2.5,0);
  \node[black, above] at (2.6,0) {$x_2=0$};
  \draw[black] (-2.5,2.5) -- (2.5,-2.5);
  \node[black, above right] at (-2.2,2.2) {$x_1+x_2=0$};
  \draw[->, black, very thick] (0,0) -- (0,1.5);
  \node[black, right] at (0,1.2) {$r_2$};
  \draw[->, black, very thick] (0,0) -- (1.5,0);
  \node[black, above] at (1.2,0) {$r_1$};
  \draw[->, black, very thick] (0,0) -- (1.5,-1.5);
  \node[black, below] at (1.1,-1.1) {$r_3$};
\end{tikzpicture}
\end{center}
In each one dimensional vector subspace defined by the hyperplanes, we 
choose a vector $r_i$. By \Cref{lem:fin_gen_set}, the external zonotopal algebra $R_Z^{\text{ext}}$ 
is the quotient of the symmetric algebra generated by the $r_i$ subject to the linear relations among the $r_i$ and the relations $r_i^{m_i + 1}$, where $m_i$ is the number of hyperplanes not containing $r_i$. 

Through direct calculation, we can compute
\[\text{Orb}(Z\cap \Z^2) \simeq R_Z^{\text{ext}} \simeq \frac{\C[r_1,r_2]}{\langle r_1^3,r_2^3, (r_1-r_2)^3\rangle } \quad\text{and}\quad i_Z(1;q)= \Hilb(R_Z^{\text{ext}};q)=1+2q+3q^2+q^3.\]
We can also compute $i_Z(1;q)$ using \Cref{thrm:q-ehrhart}:
\[i_Z(1;q) = qT_{U_{2,3}}(1+q,q^{-1})= q((q+1)^2+(1+q)+q^{-1}) = 1+2q+3q^2+q^3.\]

\noindent The graded Ehrhart polynomial $\ehr_Z(t;q) \in \mathcal{R}_q \subseteq \Q(q)[t]$ of $Z$ (\Cref{prop:ehrhart_poly}) is
\begin{gather*}
    \ehr_Z(t;q)= (q^3-q)t^3+3q^2t^2+3qt+1.
\end{gather*}
This is the quantum integer-valued polynomial defined by the property 
that $\ehr_Z([m]_q;q)=i_Z(m;q)$ for all non-negative integers $m$. 
\Cref{prop:ehrhart_poly} tells us that after applying the bar involution 
$q \mapsto q^{-1}, t \mapsto -qt$ (\Cref{diag:bar}) and multiplying by $q^{-2}$, we obtain a quantum integer-valued polynomial
\[\widetilde{\ehr_Z}(t;q)= (1-q^{-2})t^3+3q^{-2}t^2-3q^{-2}t+q^{-2} \]
such that $\widetilde{\ehr_Z}([m]_q;q)= \widetilde{i_Z}(m;q)$ for all 
positive integers $m$. For example, after plugging in $t=[1]_q=1$, we see that $\widetilde{\ehr_Z}([1]_q;q)=1=\widetilde{i_Z}(1;q)$.

The $q$-binomial coefficient polynomials $\qbinom{t}{k}\in \mathcal{R}_q\subseteq \Q(q)[t]$ are polynomials such that, after setting $t=[m]_q$, we obtain the $q$-binomial coefficients $\binom{m}{k}_q$. We can expand $\ehr_Z(t;q)$ into the $q$-binomial coefficient polynomial basis as
\begin{align*}\ehr_Z(t;q) = &\qbinom{t}{0}+(q^3+3q^2+2q)\qbinom{t}{1}+(q^6+3q^5+4q^4-2q^2)\qbinom{t}{2}\\
&+(q^9+2q^8+q^7-q^6-2q^5-q^4)\qbinom{t}{3}. 
\end{align*}
As the generating functions of $q$-binomial coefficients are well understood, we can use our expansion of $\ehr_Z(t;q)$ into binomial coefficients to compute
\[E_Z(t,q) = \sum_{m\geq 0} i_Z(m;q)t^m =\frac{-q^4t^3-(q^3+2q^2)t^2+(2q^2+q)t+1}{(1-t)(1-tq)(1-tq^2)(1-tq^3)} \,. \]
Note that, unlike in classical Ehrhart theory, our numerator can have negative coefficients. Doing a similar process for the generating function of $\widetilde{i_Z}(t;q)$, we obtain

\[\widetilde{E_Z}(t,q) = \sum_{m\geq 1} \widetilde{i_Z}(m;q)t^m =\frac{(-q^4t^3-(q^3+2q^2)t^2+(2q^2+q)t+1)t}{(1-t)(1-tq)(1-tq^2)(1-tq^3)} \,. \]

\noindent With these explicit presentations, one can compute that $q^2 \widetilde{E_Z}(t,q) = (-1)^3 E_Z(t^{-1},q^{-1})$. In \Cref{thrm:ehrhart_series}, we prove that such a reciprocity holds for all unimodular zonotopes.

We now describe a presentation of the harmonic algebra $\H_Z$ following \Cref{prop:harmonic-pres}. The only circuit $C$ of $U_{2,3}$ is $C=123$. The complement of $C$ is the empty set. Thus
\begin{align*}
\H_Z &\simeq  \frac{\C[z_\emptyset, z_1,z_2,z_3,z_{12},z_{13},z_{23},z_{123}] }{\langle z_1+z_2-z_3, z_Sz_T - z_{S\cup T} z_{S\cap T}: S,T\subseteq [3]\rangle } \,.
\end{align*}
Using the single linear relation, we can verify that $(\H_Z)_{(1,*)}$  has $\C$-basis \begin{equation}\label{eq:harmEx}(\H_Z)_{(1,*)} = \text{span}_\C\{z_\emptyset, z_1, z_2,z_{12},z_{13},z_{23},z_{123}\}\,.\end{equation}
As each variable $z_S$ has degree $(1,|S|)$, we can use our $\C$-basis to confirm that $i_Z(1;q)$ is equal to $\Hilb((\H_Z)_{(1,*)};q)$. \Cref{prop:harmonic-pres} tells us that this equality continues to hold upon replacing $1$ with any non-negative integer $m$.

The interior ideal $\widetilde{\H_Z}$ can be identified with the ideal $\langle z_{\emptyset}\rangle$ of our presentation of $\H_Z$. As an example of this, we claim that $\widetilde{\H_Z}_{(2,*)}$ has $\C$-basis
\begin{equation}\label{eq:internalEx}\widetilde{\H_Z}_{(2,*)}= \text{span}_\C\{z_{\emptyset}^2, z_1z_{\emptyset}, z_2z_{\emptyset}, z_{12}z_{\emptyset},z_{13}z_{\emptyset},z_{23}z_{\emptyset},z_{123}z_{\emptyset}\}\,. \end{equation}

Note that our basis (\ref{eq:internalEx}) of $(\widetilde{\H_Z})_{(2,*)}$ is simply our basis (\ref{eq:harmEx}) of $(\H_Z)_{(1,*)}$ multiplied by $z_{\emptyset}$. This is a relic of the fact that $\langle z_{\emptyset} \rangle \simeq \H_Z(1,0)$ as bigraded $\H_Z$-modules. The matroid $U_{2,3}$ has one connected component, namely its ground set, and this connected component is its unique circuit. Because of this, \Cref{thrm:gor} tells us that $\H_Z$ is Gorenstein, which is consistent with our above observation. The numerator of $E_Z(t,q)$ is 
\[-q^4t^3-(q^3+2q^2)t^2+(2q^2+q)t+1. \]
Note the symmetry of its $t$-coefficients. Namely, if we let $g_k(q)$ be the coefficient of $t^k$, then $g_k(q) = -q^4 g_k(q^{-1})$ for $k\in \{0,1\}$. This symmetry is consistent with the quantum symmetry prescribed in \Cref{prop:palin}.

\section{Background}\label{sec:background}
\subsection{Graded Ehrhart theory}\label{sec:gradedEhrhart}
Recently, Reiner and Rhoades have introduced a new $q$-analogue of Ehrhart theory \cite{RR24}. For every lattice polytope $P$ and non-negative integer $m$, they construct polynomials $i_P(m;q)$ and $\widetilde{i_P}(m;q)$ in $\Z[q]$ with non-negative coefficients such that, in the $q$ equal to one limit, $i_P(m;1) = \vert mP \cap \mathbb{Z}^d\vert$ and $\widetilde{i_P}(m;1) = \vert \text{int}(mP) \cap \mathbb{Z}^d\vert$ where $\text{int}(mP)$ is the interior of the $m$th dilate of $P$. Reiner and Rhoades construct these polynomials using the methods of \emph{orbit harmonics}. The orbit harmonics ring $\text{Orb}({\mathcal{Z}})$ of a finite subset $\mathcal{Z} \subseteq \mathbb{C}^d$ is the ring
\[\text{Orb}({\mathcal{Z}})= \frac{\mathbb{C}[x_1,x_2,\ldots,x_d]}{\gr I(\mathcal{Z})}\]
where $\gr I(\mathcal{Z})$ is the ideal generated by all of the top degree homogeneous components $f_k$ of polynomials $f = f_k+ f_{k-1}+\ldots + f_0$ vanishing on $\mathcal{Z}$. The ring $\text{Orb}({\mathcal{Z}})$ is graded by degree and has total dimension equal to the cardinality of $\mathcal{Z}$. In this paper our loci $\mathcal{Z}$ will be the lattice point of real polytopes. We will view $\mathcal{Z}$ as a point locus in $\C^d$ via the canonical inclusion $\mathcal{Z} \subseteq \Z^d \hookrightarrow \Z^d\otimes_\Z \C \simeq \C^d$. 

For a $\mathbb{N}^k$-graded $\C$-algebra $R = \oplus_{\alpha\in \mathbb{N}^k} R_\alpha$, the Hilbert series of $R$ is the generating function $\Hilb(R;q_1,\ldots, q_k)\coloneqq \sum_{\alpha \in \mathbb{N}^k} \dim_\C (R_\alpha) q_1^{\alpha_1}\cdots q_k^{\alpha_k}$. If $R$ is finite dimensional as a $\C$-vector space, then $\Hilb(R;q_1,\ldots, q_k)$ is a polynomial in $\Z[q_1,\ldots, q_k]$. 

\begin{defn}\label{defn:ehr}
The graded lattice point counts $i_P(m;q)$ and $\widetilde{i_P}(m;q)$ are the Hilbert series
\[i_P(m;q) \coloneqq \Hilb(\text{Orb}(mP\cap \Z^d);q) \quad \text{and}\quad \widetilde{i_P}(m;q) \coloneqq \Hilb(\text{Orb}(\text{int}(mP)\cap \Z^d);q). \]
The graded Ehrhart series $E_P(m;q)$ and $\widetilde{E_P}(m;q)$ are the bivariate generating functions
\[E_P(t,q) = \sum_{m\geq 0} i_P(m;q) t^m \quad \text{and} \quad \widetilde{E_P}(t,q) = \sum_{m\geq1} \widetilde{i_P}(m;q)t^m \,. \]
\end{defn}
There is a bigraded $\C$-algebra $\mathcal{H}_P$ and a homogeneous ideal $\widetilde{\mathcal{H}_P} \subseteq \mathcal{H}_P$ whose bigraded Hilbert series are the graded Ehrhart series $E_P(t,q)$ and $\widetilde{E_P}(t,q)$, respectively. The algebra $\mathcal{H}_P$ is called the \emph{harmonic algebra} of $P$ and the ideal $\widetilde{\mathcal{H}_P}$ is called the \emph{interior ideal} of $P$. We now work towards defining these objects. First, we recall the definition of a \emph{Macaulay inverse system}.

\begin{defn}\label{inverse-system}
Given a polynomial $f \in \C[x_1,\ldots,x_n]$, write $f(D)$ for the 
partial differential operator $f(\frac{\partial}{\partial x_1},\ldots, 
\frac{\partial}{\partial x_n})$.
Given an ideal $I \subseteq \C[x_1,\ldots,x_n]$, let $I^\perp$ denote 
the Macaulay inverse system:
\[
  I^\perp := \{f \in \C[x_1,\ldots, x_d] \mid 0 = (f(D) \cdot g)\vert_{x_1=\ldots x_n=0}\ 
  \text{for all $g \in I$}\}.
\]
\end{defn}
For a lattice polytope $P \subseteq \R^d$, we set
\[
  V_P \coloneqq (\gr I(P \cap \Z^d))^\perp\quad \text{and} \quad\widetilde{V_P} \coloneqq (\gr I(\text{int}(P) \cap \Z^d))^\perp.
\]
  If $I \subseteq \C[x_1,\ldots,x_d]$ is a graded ideal, then $I^\perp$ is a graded vector space which is naturally isomorphic to the linear dual of the quotient by $I$. This ensures that \[i_P(m;q) = \sum_{j\geq 0} \dim(V_{mP})_j q^j \quad \text{ and } \quad \widetilde{i_P}(m;q) = \sum_{j\geq 0} \dim(\widetilde{V_{mP}})_j q^j.\]

\begin{defn}\label{defn:harm}
For a lattice polytope $P\subseteq \R^d$, consider the bigraded polynomial ring $\C[x_0,\ldots,x_d]$ where $\deg(x_0)= (1,0)$ and $\deg(x_i) = (0,1)$ for all $i\geq 1$. The harmonic algebra $\H_P$ and the interior ideal $\widetilde{\H_P}$ are the bigraded subspaces of $\C[x_0,\ldots, x_d]$ defined as
\[ \H_P \coloneqq \bigoplus_{m \geq 0} \C\cdot x_0^m \otimes_\C V_{mP} \quad \text{ and } \quad \widetilde{\H_P} \coloneqq \bigoplus_{m \geq 1} \C\cdot x_0^m \otimes_\C \widetilde{V_{mP}}\,.\] 
\end{defn}    
  
Our construction ensures that
\[\Hilb(\H_P; t,q) \coloneqq \sum_{m,j} \dim(\H_P)_{m,j}t^mq^j = E_P(t,q) \]
and similarly that $\Hilb(\widetilde{\H_P};t,q) = \widetilde{E_P}(t,q)$. 
One of the main contributions of Rhoades and Reiner is to show that $\H_P$ admits the structure of a $\C$-algebra and that $\widetilde{\H_P}$ includes as an ideal under this structure.

\begin{thrm}[\cite{RR24}]
  The harmonic algebra $\H_P$ has the structure of a ring, where the multiplication is 
  $( x_0^i\otimes f ,  x_0^j\otimes g ) \mapsto x_0^{i+j} \otimes fg $. 
  Moreover, the interior ideal $\widetilde{\H}_P$ is naturally a subspace of $\H_P$ and forms an 
  ideal under this product structure.
\end{thrm}

\subsection{Matroids and the Tutte polynomial}\label{subsec:matroids}

Let $\FF$ be a field. Given $S \subseteq [n]$, write $\FF^S$ for the 
corresponding coordinate subspace and $\pr_S:\FF^n \to \FF^S$ for the 
coordinate projection.

Suppose that $L \subseteq \FF^n$ is a $d$-dimensional linear subspace. The 
\textbf{matroid} $M$ associated to $L$ can be 
defined via its bases, which are given by
\[\mathcal{B}(M) \coloneqq \left\{ B \in \binom{[n]}{d} \,\middle\vert\, \pr_B(L) = \mathbb{F}^B \right\}.\]
The set $[n]$ of coordinates is part of the data of $M$, and is referred 
to as the \textbf{ground set}.
A matroid arising from a subspace $L \subseteq \FF^n$ is called 
\textbf{realizable}. For the defining axioms of a matroid and background on matroid theory, we refer to 
\cite{oxleyMatroidTheory2011}.

\begin{defn}
Let $M$ be a matroid realized by a linear subspace $L \subseteq 
\FF^n$, and $i \in [n]$.
\begin{itemize}
    \item We say $i$ is a \textbf{loop} if no basis of $M$ 
        contains $i$, or equivalently, $L \subseteq \FF^{[n]\setminus i}$.
    \item We say $i$ is a \textbf{coloop} if every basis of $M$ 
        contains $i$, or equivalently, $\FF^{\{i\}} \subseteq L$.
    \item If $i$ is not a loop, the \textbf{contraction} of $M$ by $i$ 
        is the matroid on $[n] \setminus i$ defined by
        \[
            \mathcal{B}(M/i) := \{ B\setminus i \mid i \in B \in 
            \mathcal{B}(M)\}\quad \text{and realized by} \quad L/i := L \cap 
            \FF^{[n]\setminus i}.
        \]
    \item If $i$ is not a coloop, the \textbf{deletion} of $M$ by $i$ is 
        the matroid on $[n] \setminus i$ defined by
        \[
            \mathcal{B}(M\setminus i) := \{ B \mid i \not\in B \in 
            \mathcal{B}(M)\}\quad \text{and realized by} \quad L\setminus i := 
            \pr_{[n]\setminus i}(L).
        \]
\end{itemize}
\end{defn}

The Tutte polynomial is the universal matroid invariant which satisfies 
the following deletion/contration recursion. We use 
\cite{welshTuttePolynomial1999} as our reference for the following well 
know statement.
\begin{prop}\label{prop:tutte_def}
    There exists a unique way of assigning a polynomial $T_M(x,y) \in 
    \Z_{\geq 0}[x,y]$ to every matroid, such that $T_M(x,y) = x^ay^b$ whenever 
    the ground set of $M$ consists of $a$ coloops and $b$ loops, and 
    $T_M(x,y) = T_{M\setminus i}(x,y) + T_{M/i}(x,y)$ whenever $i$ is neither a loop nor a coloop. 
    The polynomial $T_M(x,y)$ has non-negative integer coefficients and is of $x$-degree at most $d$ and $y$-degree at most $n-d$.
\end{prop}

We now review the $m$-thickening of a matroid with ground set $[n]$, which 
is a matroid with ground set $[m] \times [n]$. We will find the 
following notation useful.
\begin{notation}\label{zero-one-matrix-notation}
We consider a subset $S \subseteq [m] 
\times [n]$ as an $m \times n$ matrix with $0/1$ entries, where the 
$(i,j)$th entry is a one if and only if $(i,j) \in S$. We refer to 
\[
    S_{i,*} \coloneqq \{j \in [n] \mid (i,j) \in S\} \quad \text{and} \quad S_{*,j} \coloneqq \{i \in [m] \mid (i,j) \in S\}
\]
as the $i$th row of $S$ and $j$th column of $S$, respectively. We refer to 
\[
    \rowsupp(S) := \{i \in [m] \mid |S_{i,*}|>0\}
\]
as the row support of $S$, and we refer to 
\[
    \colsupp(S) := \{j \in [n] \mid |S_{*,j}|>0\}
\]
as the column support of $S$. See \Cref{matrix-notation-example}.
\end{notation}

\begin{example}\label{matrix-notation-example}
    The subset $S = \{(1,4),(1,5),(2,2),(2,3),(2,4)\} \subseteq [3] 
    \times [5]$ is represented by the matrix
    \[
\begin{pmatrix}
    0 & 0 & 0 & 1 & 1 \\
    0 & 1 & 1 & 1 & 0 \\
    0 & 0 & 0 & 0 & 0
\end{pmatrix}\,.
\]
We have that
    $\rowsupp(S) = \{1,2\}$, $\colsupp(S) = \{2,3,4,5\}$, $S_{1,*} = 
    \{4,5\}$, and $S_{*,4}= \{1,2\}$.
\end{example}

\begin{defn}\label{dilation-def}
    Let $M$ be a matroid of rank $d$ with ground set $[n]$ realized by 
    $L \subseteq \FF^n$. The \textbf{$m$-thickening}
    $M(m)$ is the matroid of rank $d$ with ground set $[m] \times [n]$ defined by
    \[
        \mathcal{B}(M(m)) := \left\{ S \in {\binom{[m] \times [n]}{d}} 
            \,\middle\vert\, \colsupp(S) \in 
                \mathcal{B}(M)\right\},
    \]
    and realized by the image $L(m)$ of $L$ under the diagonal embedding $\FF^n \subseteq 
(\FF^n)^m$.
\end{defn}

The following formula can be found in \cite[Lemma 
2.2]{bekeMerinoWelshConjectureFalse2024} or \cite{JVW90}.
\begin{prop}\label{prop:thickening}
    If $M$ is a matroid of rank $d$ and $M(m)$ is the $m$-thickening, then  
    \[
    T_{M(m)}(x,y) = (y^{m-1} + y^{m-2} + \ldots + 1)^{d}T_M(\frac{y^{m-1} + y^{m-2} + \ldots + y + x}{y^{m-1} + y^{m-2} + \ldots + y + 1}, y^m).
    \]
\end{prop}

\subsection{Zonotopes and zonotopal algebras} \label{subsec:zonotopes}

\begin{defn}
A \emph{zonotope} $Z\subseteq \R^d$ is a polytope that can be written as 
the image of the unit cube $[0,1]^n\subseteq \R^n$ under a linear 
projection $A:\R^n\twoheadrightarrow \R^d$. We often take $A$ as part of the data of 
$Z$ and write $Z= A\cdot [0,1]^n \subseteq \R^d$. 
If $A$ can be written 
as a $d \times n$ full rank integer matrix which is unimodular, 
i.e.\ all maximal minors of $A$ are contained in the set $\{-1,0,1\}$, then we 
say that $Z$ is a \emph{unimodular zonotope}. 
\end{defn}

We refer the reader to \cite{BVM18} for an excellent survey and introduction to lattice zonotopes. The Ehrhart theory of unimodular zonotopes is entirely governed by 
matroid theory. The following theorem of Stanley illustrates this phenomenon. 

\begin{thrm}[Theorem 2.2 of \cite{stanleyzonotope}] \label{thrm:stanley_ehrhart}
    Let $m\geq 1$ be an integer, $Z=A\cdot [0,1]^n\subseteq \R^d$ be a unimodular zonotope and $M$ be the matroid of $A$. For any positive integer $m$,
  \begin{equation*}
    \vert mZ \cap \Z^d \vert  =  m^d 
    T_M(\frac{m+1}{m},1)\,\,\,\text{and}\,\,\,
    \vert \mathrm{int}(mZ) \cap \Z^d \vert = m^d 
    T_M(\frac{m-1}{m},1).
\end{equation*}
In particular, 
\[T_M(2,1) = |Z \cap \Z^d|\quad \text{and} \quad T_M(0,1) = 
    |\mathrm{int}(Z) \cap \Z^d|. \]
\end{thrm}

\begin{remark}
    Stanley's proof of \Cref{thrm:stanley_ehrhart} relies on the construction of a half open decomposition of a zonotope and the valuativity of lattice point counts. In contrast, graded lattice point counts are \emph{not} valuative \cite[Section 5.1]{RR24} and our proof of \Cref{thrm:q-ehrhart} does not make use of a half of open decomposition. Backman, Baker, and Yuen \cite{BBY19} give an alternative matroid-theoretic proof of \Cref{thrm:stanley_ehrhart} by giving an explicit combinatorial interpretation of the lattice points in a zonotope. It will be interesting to compare their methods with ours.
\end{remark}

\begin{remark}\label{dilated-zonotope-vs-matroid}
    Suppose that $Z = A \cdot [0,1]^n$ is a unimodular zonotope, and $L = 
    \rowspace(A)$. Let $A(m)$ be the matrix with columns indexed by $[m] 
    \times [n]$, whose $(i,j)$th column is the $j$th column of $A$. In this case,
    \[
        mZ = A(m) \cdot [0,1]^{mn}\quad \text{and} \quad L(m) = 
        \rowspace(A(m)).
    \]
    One can check that the matrix $A(m)$ is unimodular, and therefore 
    $mZ$ is a unimodular zonotope.

   Combining this observation with \Cref{thrm:stanley_ehrhart}, we see that 
   \begin{equation}\label{eq:thickening_eq}T_{M(m)}(2,1) = m^d 
    T_M(\frac{m+1}{m},1) \quad \text{and} \quad T_{M(m)}(0,1) = m^d 
    T_M(\frac{m-1}{m},1). \end{equation}
    In fact, \Cref{eq:thickening_eq} holds for all matroids $M$. This can be deduced from\Cref{prop:thickening}.
\end{remark}

We now review zonotopal algebras. Suppose that $L \subseteq \C^n$ 
is a linear subspace.

\begin{defn}\label{defn:zon}
    Given an element $v \in L$, write $m(v)$ for the number of coordinate 
    functions of $\C^n$ which are nonzero on $v$.
  Define an ideal $J_{L,k} := \langle v^{m(v)+k} \mid 0 \not=v \in L 
  \rangle \subseteq \Sym L$, and write $R_{L,k} := \Sym L / J_{L,k}$. We call 
  $R_{L,-1}$ the internal zonotopal algebra, $R_{L,1}$ the external 
  zonotopal algebra, and we write
  \[
      R_L^{\text{ext}} := R_{L,1},\ J_L^\text{ext} := J_{L,1},\ R_L^{\text{int}} := 
      R_{L,-1}\ \text{and}\ J_L^{\text{int}} := J_{L,-1}.
  \]
  When $L$ is the complexified rowspace of $A$, where $Z = A \cdot 
  [0,1]^n$ is a zonotope, we will write $Z$ in place of $L$ in the 
  above.
\end{defn}

\begin{remark}
The algebra $R_{L,1}$ appeared independently in 
\cite{postnikovAlgebrasCurvatureForms1999} and 
 \cite{wagnerAlgebraFlowsGraphs1998, wagnerAlgebrasRelatedMatroids1999}, 
and is also known as the Postnikov-Shapiro 
algebra. The 
ideal $J_{L,0}$ and its inverse system first appeared in approximation 
theory in 
\cite{deboorBsplinesParallelepipeds1982, deboorTwoPolynomialSpaces1991}. A 
theory containing the three cases $k=-1,0,1$ was developed in \cite{HR}, where the terms ``internal,'' 
``central,'' and ``external'' zonotopal algebra were coined. The family 
of algebras for arbitrary $k$ was studied by \cite{AP10} concurrently with 
\cite{HR}. See also the correction 
\cite{ardilaCorrectionCombinatoricsGeometry2015}. The central zonotopal algebra does not appear in our current work.
\end{remark}

\begin{thrm}[\cite{AP10,HR}\protect\footnotemark]\label{zonotopal-hilbert}
If $T_M(x,y)$ denotes the Tutte polynomial of the matroid $M$ of $L$, then
\[
\Hilb(R^\text{ext}_L, q) = q^{n-d}T_M(1+q,q^{-1})\quad \text{and} \quad 
\Hilb(R^\text{int}_L, q) = q^{n-d}T_M(0,q^{-1}).
\]
\end{thrm}

\footnotetext{Both formulas are claimed in {\cite{AP10}}, however the 
proof is faulty in the case of $R_Z^\text{int}$. It is pointed out in 
the correction {\cite{ardilaCorrectionCombinatoricsGeometry2015}} that 
the formula still holds by the results of {\cite{HR}}.}

\begin{remark}\label{zonotopal-point-count}
    Note that taking the $q$ equal to $1$ limit in 
    \Cref{zonotopal-hilbert} shows that $\dim 
    R_Z^{\text{ext}}$ and $\dim R_{Z}^{\text{int}}$ count lattice points 
    and, respectively, internal lattice points in a unimodular zonotope, 
    hence the names ``external'' and ``internal''.
\end{remark}

We now give finite generating sets for the ideals $J_{L}^\text{ext}$ 
and $J_{L}^\text{int}$. If $v$ is a vector in $\mathbb{C}^n$, the support of $v$ is the set $\text{Supp}(v)= \{ i \in [n]: v_i \neq 0\}$. We will write $m(v)$ for the size of the support of $v$.
\begin{defn}
    A cocircuit vector of $L \subseteq \C^n$ is a nonzero vector $v \in 
    L$ whose support is minimal amoung the supports of all nonzero 
    vectors in $L$.
\end{defn}

Note that a cocircuit vector is determined up to scaling by 
$\text{Supp}(v)$, which is a cocircuit of $M$. Therefore the following 
generating sets for $J_L^\text{ext}$ and $J_L^\text{int}$ are finite.

\begin{lemma}[{\cite[Lemma 1]{ardilaCorrectionCombinatoricsGeometry2015}}]\label{lem:fin_gen_set}
    We have that $J_{L}^\text{ext}$ is generated by $v^{m(v)+1}$ and 
    $J_{L}^\text{int}$ is generated by $v^{m(v)-1}$ where $v$ ranges 
    over all equivalence classes of cocircuit vectors up to scaling.
\end{lemma}

The following proposition (\Cref{thrm:zon_orbit}) was first proved in \cite{RRT23} for graphical 
 zonotopes, in \cite{CDBHP} for the internal case, and is implicit in the 
 main theorems of \cite{HR}. We will give a direct proof for the 
 convenience of the reader. In the following, we say that a vector $v \in \Z^n$ is primitive if $\alpha v \not\in \Z^n$ for all $0<\alpha<1$.
\begin{lemma}\label{lem:facet_ineq}
    If $Z = A\cdot[0,1]^n$ is a unimodular zonotope, and $v \in \Z^n \cap L$ is a primitive cocircuit vector, then the inequalities
    \[
    \min \{\langle v,x\rangle \mid x \in Z\} \leq \langle v, - \rangle \quad \text{and}\quad \langle v, - \rangle \leq \max \{\langle v,x\rangle \mid x \in Z\}
    \]
    are facet inequalities for $Z$, and moreover
    \[
    \max \{\langle v,x\rangle \mid x \in Z\} - \min \{\langle v,x\rangle \mid x \in Z\} = m(v).
    \]
\end{lemma}
\begin{proof}
    By definition, $L$ is the rowspace of $A$, and $L^\vee$ is therefore identified with the column space of $A$. We therefore think of $Z \subseteq L^\vee$. The statement about facet inequalities follows from \cite[Theorem 7.16]{zieglerLecturesPolytopes1995}.
    For any $i 
    \not\in \text{Supp}(v)$, we have that $v$ is contained in the 
    hyperplane $H_i = L \cap \C^{[n]-i}$, and is therefore orthogonal to 
    the defining normal $A \cdot e_i \in L^\vee$. Therefore we may write 
    \[
        \text{Supp}(v) = S \sqcup T,\quad S = \{i \mid \langle v, A 
        \cdot e_i \rangle > 0\},\quad T = \{i \mid \langle v, A \cdot e_i 
        \rangle < 0 \}.
    \]
    The unimodularity assumption implies that $\langle 
    v, e_j \rangle \in \{-1,0,1\}$ for each $j$. See, for example, \cite[Claim 
    5.42]{tutteLecturesMatroids1965}. The vertices of $Z$ are obtained 
    as sums of subsets of the $A \cdot e_i$, so the largest possible 
    value of $\langle v, - \rangle$ on a vertex is $|S|$, the smallest 
    possible value is $-|T|$.
\end{proof}
\begin{prop}\label{thrm:zon_orbit}
    If $Z = A \cdot [0,1]^n$ is a unimodular zonotope, then
    \[
    R^\text{ext}_Z \cong \text{Orb}(Z \cap \Z^d)\quad \text{and} \quad 
R^\text{int}_Z \cong \text{Orb}(\text{int}(Z) \cap \Z^d).
    \]
\end{prop}

\begin{proof}
Let $v\in \Z^n\cap L$ be a primitive cocircuit vector. By definition, $L$ is the row space of $A$, and $L^\vee$ is therefore identified with the column space of $A$. We use this identification to view $v$ as a function on $Z$. Let $\alpha_{\min}$ and $\alpha_{\max}$ be the minimum and maximum values the function $v\in \Sym(L)$ takes on $Z$. The polynomial
\[        (v - \alpha_{\min})(v - \alpha_{\min}+1) \ldots (v+\alpha_{\max} - 1)(v + \alpha_{\max}) \in \Sym L \]
vanishes on $Z\cap \Z^d$. Therefore the initial form $v^{\alpha_{\max}-\alpha_{\min}+1}$ is a relation of $\text{Orb}(Z \cap \Z^d)$. By \Cref{lem:facet_ineq}, $\alpha_{\max}-\alpha_{\min}+1= m(v)+1$. By \Cref{lem:fin_gen_set}, this implies that $\text{Orb}(Z \cap \Z^d)$ is a quotient of $R_Z^{\text{ext}}$. By the definition of $\text{Orb}(Z\cap \Z^d)$ and by \Cref{zonotopal-point-count}, both algebras have dimension equal to $|Z \cap \Z^d|$. Therefore they are isomorphic.

The internal case is similar. By \Cref{lem:facet_ineq}, if $v\in \Z^n\cap L$ is a primitive cocircuit vector, then the polynomial
    \[        (v - \alpha_{\min}+1)(v - \alpha_{\min}+2) \ldots (v+\alpha_{\max} - 2)(v + \alpha_{\max}-1) \in \Sym L \]
    with initial form $v^{m(v)-1}$, vanishes on $\text{int}(Z) \cap 
    \Z^d$, and thus the internal case follows as in the external case.
\end{proof}

We will also need the following spanning set for the inverse system 
(\Cref{inverse-system}) of 
$J_Z^\text{ext}$, which was proved independently in \cite[Proposition 
4.21]{AP10} and \cite[Theorem 4.10(5)]{HR}. The 
spanning set itself appeared independently as the main object of study 
in \cite{bergetProductsLinearForms2010}.

\begin{prop}\label{external-gens}
    $J_{L,1}^\perp \subseteq \Sym L^\vee$ is spanned by 
    square free products $\prod_{i \in S} l_i \in \Sym L^\vee$, where $l_i: L \to \C$ is 
    the restriction of the $i$th coordinate function of $\C^n$, and $S 
    \subseteq [n]$.
\end{prop}

\subsection{Arrangement Schubert varieties}\label{sec:matroid-schubert}
\begin{defn}\label{ASV-def}
    The arrangement Schubert variety $Y_L$ of a linear subspace $L 
    \subseteq \C^n$ is the closure of $L$ under the embedding
    \[L \hookrightarrow \mathbb{C}^n \hookrightarrow (\P^1)^n\]
    where the second map is defined by sending a point $(a_1,\ldots, a_n)\in \C^n$ to the point $([a_1:1],\ldots,[a_n:1]) \in (\P^1)^n$.
\end{defn}

The arrangement Schubert variety is an important projective variety 
whose geometry is largely dictated by the matroid $M$ of $L \subseteq 
\C^n$. It was first studied concurrently by \cite{AB} and \cite{liImagesRationalMaps2018}, and was later used by \cite{huhEnumerationPointsLines2017} to prove the top heavy conjecture for realizable matroids. 

We now describe the defining multihomogeneous equations for $Y_L$ given in 
\cite{AB}.
\begin{defn}\label{circuit-equations}
    Given a circuit $C \subseteq [n]$ of $M$, there is a unique (up to 
    scaling) linear dependence $\sum_{i \in C} \alpha_{C,i}A_i = 0$.  We will fix a representative and define the 
    corresponding circuit equation as
    \[
        f_C := \sum_{i \in C} \alpha_{C,i} x_i \in \C[x_1, \ldots, x_n].
    \]
    Define the multihomogeneous circuit equations of $L$ to be
    \[
        \widetilde{f_C} := \sum_{i \in C} \alpha_{C,i} x_i \prod_{j 
        \in C\setminus i}y_j \in \C[x_1, \ldots, x_n, y_1, \ldots, y_n],
    \]
    and call the ideal $\widetilde{I_L}$ that they generate the 
    multihomogeneous circuit ideal.
\end{defn}

\begin{prop}[{\cite[Theorem 1.3(a)]{AB}}]
    The subvariety $Y_L \subseteq (\P^1)^n$ is cut out 
    by the multihomogeneous circuit equations $\widetilde{f_C}$ of $L$.
\end{prop}

\begin{remark}\label{bigrading-1}
Consider the Segre embedding $\Phi: (\P^1)^n \to \P^{2^n-1}$ defined by 
sending a point $([a_1:b_1],[a_2:b_2],\ldots, [a_n:b_n])$ to the point 
$[c_S:S\subseteq [n]]$ where $c_S = \prod_{i\in S} a_i 
\prod_{j\in [n]\setminus S}b_j$. By composing the embeddings $Y_L 
\hookrightarrow (\P^1)^n \hookrightarrow \P^{2^n-1}$, we can consider 
$Y_L$ as a projective subvariety of $\P^{2^n-1}$.

As $Y_L$ is the closure of a linear space, it is closed under the 
$\mathbb{C}^*$-action on $(\mathbb{P}^1)^n$ where $t\in \mathbb{C}^*$ 
acts by $t\cdot ([a_1:b_1],[a_2:b_2],\ldots, [a_n:b_n]) = 
([t^{-1}a_1:b_1],[t^{-1}a_2:b_2],\ldots, [t^{-1}a_n:b_n])$. 
This action restricts to the inverse scaling action of $\C^*$ on 
$L$. If we equip $\P^{2^n-1}$ with the $\mathbb{C}^*$-action defined 
by $t\cdot [c_S]_{S\subseteq [n]} = [t^{-\vert S \vert} c_S : S\subseteq 
[n]]$, then all of the above embeddings become equivariant. Therefore the homogeneous coordinate ring of $Y_L\subseteq \P^{2^n-1}$ comes equipped with a bigrading. We use this bigrading in our statement of 
\Cref{thrm:geometry}.
\end{remark}

\subsection{Quantum integer-valued polynomials}\label{subsec:quantum}

In \Cref{sec:combinatorics}, we study the graded Ehrhart series of unimodular zonotopes using the theory of quantum integer-valued polynomials. Quantum integer-valued polynomials are a natural $q$-deformation of integer-valued polynomials and were systematically introduced and studied in \cite{HH}. Here, we present the basic definitions and properties of quantum integer-valued polynomials.

Given an integer $n$, the \emph{$q$-number} $[n]_q$ is defined as the rational function $(1-q^n)/(1-q)\in \mathbb{Q}(q)$. When $n\geq 0$, we also have that $[n]_q = 1+q+\ldots+q^{n-1}$. For a positive integer $n$, let $[n]_q! \coloneq [n]_q [n-1]_q \cdots [1]_q$ and define $[0]_q! \coloneq 1$. For a positive integer $n$ and integer $k$, the $q$-binomial coefficient is defined as $\binom{n}{k}_q \coloneq \frac{[n]_q!}{[n-k]_q![k]_q!}$ when $0\leq k \leq n$ and $\binom{n}{k}_q= 0$ when $k>n$ or $k<0$.

A polynomial $f(t;q)\in \mathbb{Q}(q)[t]$ is \emph{quantum integer-valued} if $f([n]_q;q) \in \mathbb{Z}[q,q^{-1}]$ for all $n\in \mathbb{Z}$. The set of quantum integer-valued polynomials can be equipped with the structure of a $\mathbb{Z}[q,q^{-1}]$-algebra. We denote this algebra by $\mathcal{R}_q$ and refer to it as the \emph{ring of quantum integer-valued polynomials}. The set
\[\mathcal{R}_q^+\coloneqq \left\{ f(t;q)\in \mathcal{R}_q: f([n]_q;q)\in \mathbb{Z}[q] \text{ for all positive integers } n \right\} \]
forms a sub $\Z[q]$-module of $\mathcal{R}_q$. We think of $\mathcal{R}_q^+$ as the ``positive part'' of $\mathcal{R}_q$.

\begin{defn}
  For a positive integer $k$, the $q$-binomial coefficient polynomial is equal to 
  \[ \qbinom{t }{ k } \coloneq \frac{t(t-[1]_q)\cdots (t-[k-1]_q)}{q^{\binom{k}{2}}[k]_q\cdots[1]_q}\in \mathcal{R}_q. \]
  We also define $\qbinom{t}{0} \coloneq 1$.
\end{defn}

 For a non-negative integer $n$, evaluating $ \qbinom{ t}{ k }$ at $t=[n]_q$ recovers the $q$-binomial coefficient
  \[ \qbinom{ [n]_q }{ k } = \binom{n}{k}_q. \]

  A fundamental insight into the ring $\mathcal{R}_q$ is that it is a free $\mathbb{Z}[q,q^{-1}]$-module generated by the \emph{$q$-binomial coefficient polynomials}. We state this observation following \cite[Proposition 1.2]{HH}. However, we note that Harman and Hopkins credit \cite[Theorem 14]{Bha97} and \cite[Chapter 2, Exercise 15]{CC97} for the result. In turn, \cite[Chapter 2, Exercise 15]{CC97} attributes the results to \cite[Proposition 2.2]{Gra90}.

\begin{prop} {\cite[Proposition 1.2]{HH}}\label{prop:HH}
The collection $\{ \qbinom{t}{ k}: k\geq 0\}$ forms a $\mathbb{Z}[q]$-basis of $\mathcal{R}_q^+$ and also forms a $\mathbb{Z}[q,q^{-1}]$-basis of $\mathcal{R}_q$.
\end{prop}

The ring $\mathcal{R}_q$ carries a \emph{bar involution} $\overline{\,\cdot\,}: \mathcal{R}_q\to\mathcal{R}_q$ \cite[Section $6$]{HH}. The bar involution is defined as the restriction of the involution $\mathbb{Q}(q)[t]\to\mathbb{Q}(q)[t]$ which sends $q\mapsto q^{-1}$ and $t\mapsto-qt$. After taking the $q$ equal to $1$ limit, this involution specializes to the map $\Q[t]\to \Q[t]$ which send $f(t)$ to $f(-t)$. Given $f(t;q)\in \Q(q)[t]$ and $n\in \Z$, let $f([n]_q;q^{-1})$ be the element of $\mathbb{Q}(q)$ obtained by applying the involution $q\mapsto q^{-1}$ to $f([n]_q;q)\in \mathbb{Q}(q)$. The bar involution is defined so that $\overline{f}([n]_q;q)=f([-n]_q;q^{-1})$. We can visualize this relationship by the commutative diagram:
  \begin{equation}      \label{diag:bar}\begin{tikzcd}
	{\mathbb{Q}(q)[t]} & &{\mathbb{Q}(q)} \\
	{\mathbb{Q}(q)[t]} & &{\mathbb{Q}(q)}
	\arrow["{f(t;q)\mapsto \overline{f}(t;q)}"', from=1-1, to=2-1]
        \arrow["{t\mapsto [n]_q}", from=1-1, to=1-3]
	\arrow["{q\mapsto q^{-1}}", from=1-3, to=2-3]
	\arrow["{t\mapsto [-n]_q}", from=2-1, to=2-3]
      \end{tikzcd}\end{equation}

    \section{Graded Ehrhart polynomials and quantum reciprocity}\label{sec:combinatorics}
    In this section, we prove a formal reciprocity statement for quantum integer-valued polynomials. We then apply this formal reciprocity statement to the case of unimodular zonotopes. 
    \subsection{Quantum Reciprocity}
    We are interested in the generating functions of quantum integer-valued polynomials. We collect a series of results here about the behavior of such generating functions. These results, and their proofs, are analogous to the facts that the generating functions of integer-valued polynomials are rational and obey a formal reciprocity law \cite[Theorem 4.1.1 and Proposition 4.2.3]{EC1}, \cite[Section 4.1]{crt}.

    In the proofs of the next two results, we will use repeatedly use the negative $q$-binomial theorem \cite[Chapter 3, Exercise 8]{SagComb}. 

    \begin{thrm}[The Negative $q$-Binomial Theorem]\label{thrm:qbinomial}
    For a fixed positive integer $n$,
\[        \sum_{k\geq 0} \binom{n+k-1}{k}_qt^k = \frac{1}{(1-t)(1-qt)\cdots (1-q^{n-1}t)}\,.\]
\end{thrm}
In the $q$ equal to $1$ limit, the negative $q$-binomial coefficient theorem specializes to an expression for the generating function of the multiset coefficients $\binom{n+k-1}{k}= (-1)^k\binom{-n}{k}$.
After using the identity $\binom{n}{k}_q = \binom{n}{n-k}_q$ and rewriting, the negative $q$-binomial theorem gives the following corollary.
\begin{corollary}\label{cor:qbinomial}
For a fixed non-negative integer $k$,
\[\sum_{n\geq0}\binom{n}{k}_qt^n =  \frac{t^{k}}{(1-t)(1-qt)\cdots (1-q^{k}t)}\,. \]
\end{corollary}

\Cref{cor:qbinomial} tells us that the generating functions of the $q$-binomial coefficient polynomials are rational functions in $\mathbb{Q}(t,q)$. The following proposition uses this fact in order to prove that the generating function of any quantum integer-valued polynomial is a rational function in $\Q(t,q)$.
    \begin{prop}\label{lem:formal_rationality}
      Let $f(t;q)\in \mathcal{R}_q$ be of $t$-degree $d$. The generating function \[E(t,q) = \sum_{n\geq 0} f([n]_q;q) t^n\] is a rational function in $\mathbb{Q}(t,q)$. This rational function has the form
      \[\frac{N(t,q)}{\prod_{i=0}^d (1-tq^i)} \]
      where $N(t,q)\in \mathbb{Z}[t,q,q^{-1}]$ and $N(t,q)$ has $t$-degree at most $d$. If, in addition, $f(t;q)\in \mathcal{R}_q^+$, then $N(t,q)\in \mathbb{Z}[t,q]$.
\end{prop}
\begin{proof}
  Suppose $f(t;q)$ has $t$-degree $d$. By \Cref{prop:HH}, we can expand $f(t;q)$ into the $q$-binomial coefficient polynomials as $f= \sum_{k=0}^d f_k(q) \qbinom{t}{k}$ where each $f_k(q)$ is an element of $\mathbb{Z}[q,q^{-1}]$. If $f(t;q)\in \mathcal{R}_q^+$, then we know that each $f_k(q)$ is instead inside of $\mathbb{Z}[q]$. Then
  \begin{align*}
    \sum_{n\geq0}f([n]_q;q)t^n &= \sum_{n\geq 0} \sum_{k=0}^d f_k(q)\qbinom{[n]_q}{k}t^n \\
                       &= \sum_{k=0}^d f_k(q)\sum_{n\geq 0}\binom{n}{k}_qt^n\\
                       &= \sum_{k=0}^d f_k(q)\frac{t^{k}}{(1-t)(1-tq)\cdots(1-tq^{k})}\\
                       &=\frac{\sum_{k=0}^d f_k(q)t^{k}(1-tq^{k+1})\cdots(1-tq^d)}{\prod_{i=0}^d (1-tq^i)}.
  \end{align*}
  The third equality of this chain of equalities follows from \Cref{cor:qbinomial}.
\end{proof}

We now prove a quantum reciprocity statement comparing the generating functions of $f([n]_q;q)$ and $\overline f([n]_q;q)$ for a quantum integer-valued polynomial $f(t;q)$. After taking the $q$ equal to $1$ limit, quantum reciprocity specializes to the usual reciprocity between the generating functions of $f(t)$ and $f(-t)$; see, for example, \cite[Theorem 4.1.6]{crt}.
\begin{prop}
      \label{prop:q-reciprocity}
  Let $f(t;q)\in \mathcal{R}_q$ be of $t$-degree $d$. The
generating functions
  \[E(t,q) = \sum_{n\geq0}f([n]_q;q)t^n \text{\quad and \quad} \overline{E}(t,q)=
\sum_{n\geq 1} \overline{f}([n]_q;q)t^n= \sum_{n\geq 1} f([-n]_q;q^{-1})t^n\] are
both rational functions in $\mathbb{Q}(q,t)$. These rational functions
can be written as
  \[E(t,q) = \frac{N(t,q)}{(1-t)(1-tq)\cdots (1-tq^d)} \quad \text{ and } \quad
\overline{E}(t,q) = \frac{\overline{N}(t,q)}{(1-t)(1-tq)\cdots (1-tq^d)} \]
where $N(t,q),\overline{N}(t,q)\in \Z[t,q,q^{-1}]$ are of $t$-degree at most $d$. Furthermore, $E(t,q)$ and $\overline{E}(t,q)$ are related by
$\overline{E}(t,q)= -E(t^{-1},q^{-1})$.
    \end{prop}
    As part of the proof of \Cref{prop:q-reciprocity}, we will use the following lemma.
    \begin{lemma}{\cite[Proposition 6.3]{HH}}
      \label{lem:qbinom}
      For all non-negative integers $n$ and $k$,
      \[\overline{\qbinom{[n]_q}{k}} = (-1)^kq^{\binom{k+1}{2}}\binom{n+k-1}{k}_q\]
      where $\overline{\qbinom{[n_q]}{k}}$ is the polynomial $\overline{\qbinom{t}{k}}$ evaluated at $t$ equal to $[n]_q$.
    \end{lemma}
    After taking the $q$ equal to $1$ limit,  Lemma \ref{lem:qbinom} specializes to the statement that
    \[(-1)^k\binom{-n}{k} = \binom{n+k-1}{k}\,. \]
    This fact is a key ingredient in proving the formal reciprocity between the generating functions of $f(n)$ and $f(-n)$. Likewise, Lemma \ref{lem:qbinom} plays a vital role in our proof of \Cref{prop:q-reciprocity}.
    
    \vspace{1em}
    
    \noindent \textit{Proof of \Cref{prop:q-reciprocity}.} By \Cref{prop:HH}, we can write $f(t;q)$ as  $f(t;q) = \sum_{k=0}^d f_k(q)\qbinom{t}{k}$ where $f_k(q)\in \Z[q,q^{-1}]$. By \Cref{lem:formal_rationality}, and its proof, $E(t,q)$ is the rational function
\[E(t,q)= \sum_{k=0}^d f_k(q)\frac{t^{k}}{(1-t)(1-tq)\cdots(1-tq^{k})}.\]

We now focus on the rationality of $\overline{E}(t,q)$ and its relation with $E(t,q)$. Let $f_k(q^{-1})\in \Q(q)$ be the result of applying $q\mapsto q^{-1}$ to $f_k(q)$. By \Cref{lem:qbinom},
    \[\overline{f}([n]_q;q)= \sum_{k=0}^df_k(q^{-1}) \overline{\qbinom{[n]_q}{k}}= \sum_{k=0}^d(-1)^kq^{\binom{k+1}{2}}f_k(q^{-1})\binom{n+k-1}{k}_q. \]
  Thus we can rewrite the generating function $\overline{E}(t,q)$ as
    \begin{align*}
      \overline{E}(t,q) =\sum_{n\geq 1}\overline{f}([n]_q;q) t^n
      &= \sum_{n\geq 1}\left( \sum_{k=0}^d(-1)^kq^{\binom{k+1}{2}}f_k(q^{-1})\qbinom{n+k-1}{k}_q\right)t^n\\
      &= \sum_{k=0}^d(-1)^kq^{\binom{k+1}{2}}f_k(q^{-1})\sum_{n\geq 1}\qbinom{n+k-1}{k}_q t^n\\
            &= \sum_{k=0}^d(-1)^kq^{\binom{k+1}{2}}f_k(q^{-1})t\sum_{n-1\geq 0}\qbinom{(k+1)+(n-1)-1}{n-1}_q t^{n-1}\\
      &= \sum_{k=0}^d(-1)^kq^{\binom{k+1}{2}}f_k(q^{-1})\frac{t}{(1-t)(1-tq)\cdots(1-tq^k)}\\
      &= -\sum_{k=0}^df_{k}(q^{-1})\frac{t^{-k}}{(1-t^{-1})(1-t^{-1}q^{-1})\cdots(1-t^{-1}q^{-k})}\\
      &= -E(t^{-1},q^{-1}).
    \end{align*}
    The fourth equality of this chain of equalities follows from \Cref{thrm:qbinomial}.
    \qed

\subsection{Graded Ehrhart Polynomials}\label{subsec:ehrhart-poly}    
Let $Z=A\cdot [0,1]^n\subseteq \R^d$ be a unimodular zonotope. We now prove that the graded lattice point counts of $Z$ are evaluations of a quantum integer-valued polynomial $\ehr_Z(t;q)$. The polynomial $\ehr_Z(t;q)$ obeys a $q$-analogue of Ehrhart--Macdonald reciprocity.  These results, along with the machinery developed in the previous subsection, imply strong structural properties of the graded Ehrhart series of $Z$.

First, we show how to prove the formula for $i_Z(m;q)$ and $\widetilde{i_Z}(m;q)$ given in the introduction.
\begin{proof}[Proof of \Cref{thrm:q-ehrhart}] Let $Z = A \cdot [0,1]^n$ be a unimodular zonotope, and let $L\subseteq \C^n$ be the rowspace of $A$ over $\C$, with associated matroid $M$. For the first statement, we have 
    \begin{align*}
        i_{Z}(m;q) &= \Hilb(\text{Orb}(mZ \cap \Z^d), q) &&\text{by \Cref{defn:ehr},}\\
        &= \Hilb(R_{mZ}^\text{ext}, q) &&\text{by \Cref{thrm:zon_orbit},}\\
        &= \Hilb(R_{L(m)}^\text{ext}, q) &&\text{by \Cref{dilated-zonotope-vs-matroid},}\\
        &= q^{mn-d}T_{M(m)}(1+q,q^{-1}) &&\text{by \Cref{zonotopal-hilbert},}\\
        &= q^{mn-d}(q^{1-m}[m]_q)^d T_M\left(\frac{q^{1-m}[m]_q + q}{q^{1-m}[m]_q}, q^{-m}\right)
        &&\text{by \Cref{prop:thickening},}\\
        &= q^{(n-d)m}[m]_q^d T_M\left(\frac{[m+1]_q}{[m]_q},q^{-m}\right).
    \end{align*}
    The proof of the second statement about $\widetilde{i_Z}(m;q)$ follows in exactly the same way.
\end{proof}

\begin{prop}\label{prop:ehrhart_poly}
    There is a quantum integer-valued polynomial $\ehr_Z(t;q)\in \mathcal{R}_q^+$ of $t$-degree $n$ such that for all non-negative integers $m$, $\ehr_Z([m]_q;q)= i_Z(m;q)$. This polynomial satisfies a $q$-analogue of Ehrhart--Macdonald reciprocity: For all positive integers $m$,
    \[(-1)^dq^{-d}\overline{\ehr_Z}([m]_q;q)= (-1)^dq^{-d} \ehr_Z([-m]_q;q^{-1}) = \widetilde{i_Z}(m;q)\,. \]
  \end{prop}
  \begin{proof}
    We first construct the polynomial $\ehr_Z(t;q)\in \mathbb{Q}(q)[t]$ such that $\ehr_Z([m]_q;q) = i_Z(m;q)$ for all non-negative integers $m$. Define
\[\ehr_Z(t;q) \coloneqq (1+(q-1)t)^{n-d}t^d T_M(\frac{qt+1}{t}, \frac{1}{1+(q-1)t}). \]
The $x$ and $y$ degrees of $T_M(x,y)$ are bounded above by $d$ and $n-d$, respectively. This implies that $\ehr_Z(t;q)$ is an honest element of $\Q(q)[t]$. Furthermore, as $i_Z(m;q)\in \mathbb{Z}[q]$ for all non-negative integers $m$, \cite[Lemma 4.2]{HH} implies that $\ehr_Z(t;q)\in \mathcal{R}_q^+$. Using \Cref{thrm:q-ehrhart}, we calculate
\begin{align*}
  \ehr_Z([m]_q;q) &= (1+(q-1)[m]_q)^{n-d}[m]_q^d T_M(\frac{q[m]_q+1}{[m]_q}, \frac{1}{1+(q-1)[m]_q})\\
               &= q^{m(n-d)}[m]_q^d T_M(\frac{[m+1]_q}{[m]_q}, q^{-m})\\
               &= i_Z(m;q)\,.
\end{align*}
We now focus on proving the reciprocity statement. That \[(-1)^dq^{-d}\overline{\ehr_Z}([m]_q;q)= (-1)^dq^{-d} \ehr_Z([-m]_q;q^{-1})\] follows directly from the definition of the bar involution (\Cref{diag:bar}). Again, by the definition of the bar involution, we can compute
\begin{align*}
  (-1)^dq^{-d}\overline{\ehr_Z}(t;q) &= (-1)^dq^{-d} (1-(q^{-1}-1)qt)^{n-d}(-qt)^d T_M(\frac{-t+1}{-qt}, \frac{1}{1-(q^{-1}-1)qt})\\
   &= (1+(q-1)t)^{n-d}t^d T_M(\frac{q^{-1}(t-1)}{t}, \frac{1}{1+(q-1)t}).
\end{align*}
From this computation and \Cref{thrm:q-ehrhart}, it follows that
\[(-1)^dq^{-d}\overline{\ehr_Z}([m]_q;q) = q^{m(n-d)}[m]_q^d 
T_M(\frac{[m-1]_q}{[m]_q}, q^{-m}) = \widetilde{i_Z}(m;q).\qedhere\]
  \end{proof}

Combining \Cref{prop:ehrhart_poly} and \Cref{prop:q-reciprocity}, we obtain the following theorem.  

\begin{thrm}\label{thrm:ehrhart_series}
    The Ehrhart series $E_Z(t,q)$ and $\widetilde{E_Z}(t,q)$ are rational function in $\Q(t,q)$ and can be written as
    \[E_Z(t,q) = \frac{N_Z(t,q)}{(1-t)(1-tq)\cdots (1-tq^n)} \quad \text{ and } \quad \widetilde{E_Z}(t,q) = \frac{\widetilde{N_Z}(t,q)}{(1-t)(1-tq)\cdots (1-tq^n)}  \]
  where $N_Z(t,q),\widetilde{N_Z}(t,q)\in \Z[t,q]$ are of $t$-degree at most $n$. These functions exhibit q-analogue of Ehrhart--Macdonald reciprocity, namely,
\[q^d \widetilde{E_Z}(t,q) = (-1)^{d+1}E_Z(t^{-1}, q^{-1}).\]
\end{thrm}
\begin{proof}
Once we know that $E_Z(t,q)$ and $\widetilde{E_Z}(t,q)$ are the generating functions of elements in $\mathcal{R}_q^+$, the first statement follows from \Cref{lem:formal_rationality}. The $E_Z(t,q)$ case is handled directly by \Cref{prop:ehrhart_poly}. \Cref{prop:ehrhart_poly} also tells us that the evaluation of $(-q)^{-d} \overline{\ehr_Z}(t;q)$ at $t=[m]_q$ is equal to $\widetilde{i_Z}(m;q)$ for all positive integers $m$. By \cite[Proposition 5.1]{HH}, there is a quantum integer-valued polynomial $g(t;q) \in \mathcal{R}_q$ such that 
\[g([m]_q;q)= (-q)^{-d} \overline{\ehr_Z}([m+1]_q;q)\]
for all $m\in \Z$. Our construction ensures that $g(t;q)\in \mathcal{R}_q^+$ and that
\[\sum_{m\geq 0 } g([m]_q;q)t^m  = \widetilde{E_Z}(t,q). \]
Thus the first claim now follows.

The second claims follows from the reciprocity statements of \Cref{prop:q-reciprocity} and \Cref{prop:ehrhart_poly}.
\end{proof}

\section{Harmonic algebras and arrangement Schubert varieties}\label{sec:algebra}
We now move to establish the algebraic results in 
\Cref{thrm:geometry}, \Cref{thrm:CM-intro} and \Cref{thrm:pres-intro}. 
For the rest of this section, let $L \subseteq \C^n$ be a linear 
subspace, $Y_L$ be the arrangement Schubert variety (\Cref{ASV-def}), 
and $M$ be the matroid on $[n]$ associated to $L$. Note that, unless we specify it, we do not assume that $L$ is the linear subspace associated to a unimodular zonotope.

\begin{defn}
    Let $\O_{Y_L}(m, \ldots, m)$ denote the restriction to $Y_L$ of 
    $\O_{(\P^1)^n}(m, \ldots, m)$, where the latter is 
    obtained as the tensor product of the line bundles $\O_{\P^1}(m)$ 
    pulled back from each $\P^1$ factor. When there is no possibility of 
    confusion, we may simply write $\O(m,\ldots,m)$.
\end{defn}

Because $Y_L$ is defined as the 
closure of a linear space, it is a multiplicity free subvariety 
of $(\P^1)^n$, meaning that the $\N^n$-graded multidegree of $Y_L$ is a 
polynomial with all nonzero coefficients equal to one. We use \cite[Theorem 
1.3(c)]{AB} and \cite[Corollary 3.8]{liImagesRationalMaps2018} as 
references for this fact. Therefore we may conclude the following theorem as a special case of a more general result of Brion \cite[Theorem 1]{Brion}. To see that Brion's result applies, notice that $(\P^1)^n$ is the full flag variety of the semi-simple group $SL_2(\C)^n$.
\begin{thrm}\label{multiplicity-free-thm}
    The restriction map $H^0((\P^1)^n, \O_{(\P^1)^n}(1,\ldots,1)) \to H^0(Y_L, 
    \O_{Y_L}(1,\ldots,1))$ is surjective. The variety $Y_L$ is projectively normal and arithmetically Cohen-Macaulay with respect to $\O_{Y_L}(1, \ldots, 
    1)$. 
\end{thrm}

\begin{remark}
    The special properties of multiplicity free subvarieties, including 
    \Cref{multiplicity-free-thm} and more generalhave been recently 
    applied in the study of matroid theory. See, for example, 
    \cite{bergetEquivariantKtheoryClasses2022, 
    bergetExternalActivityComplex2025, 
    EL23, EFL25,liuVanishingTheoremsWonderful2025}.
\end{remark}

We begin by exploring the many consequences of 
\Cref{multiplicity-free-thm}. Firstly, we note that projective normality 
implies that $Y_L$ is normal as a variety, so the sections of any line 
bundle may be represented as rational functions. As we will see in the 
following remark, the situation for $Y_L$ is even better, because it 
contains affine space as an open subset.

\begin{remark}\label{sections-subspace}
    Sections of line bundles on $Y_L$ admit a concrete interpretation as 
    polynomials on $L$. By 
    \cite[Theorem 2.5]{hassettGeometryEquivariantCompactifications1999}, 
    every line bundle $\mathcal{L}$ is represented uniquely by a divisor 
    $D$ supported 
    on the boundary $Y_L \setminus L$. A section $s \in H^0(Y_L, \mathcal{L})$ may be represented as a 
    rational function $f$ on $L$ such that $\text{div}(f) + D \geq 0$. By construction, $f$ has no 
    poles on $L$ and must be a polynomial. Therefore we can consider $H^0(Y_L, 
    \O_{Y_L}(m, \ldots, m)) \subseteq \Sym L^\vee$. 
\end{remark}

Next, we note that the arithmetic Cohen-Macaulay property in 
\Cref{multiplicity-free-thm} implies that $Y_L$ is Cohen-Macaulay as an 
algebraic variety, a fact which was first proved concurrently in 
\cite{AB,liImagesRationalMaps2018}. We will make use of 
the fact that $Y_L$ is Cohen-Macaulay in several ways. Firstly, recall the following well-known 
statement, for which we refer to \cite[Definition 8.0.20 and 
Theorem 9.0.9]{coxToricVarieties2011}.

\begin{prop}\label{canonical-sheaf}
    If $X$ is a normal and Cohen-Macaulay algebraic variety, then $X$ 
    admits a dualizing sheaf $\omega_X$, whose sections over a smooth 
    open set $U \subseteq X$ are top-degree algebraic differential 
    forms.
\end{prop}

\begin{remark}\label{forms-subspace}
By restricting to the open set $L \subseteq Y_L$, \Cref{canonical-sheaf} implies that elements of $H^0(Y_L, \O(m, \ldots, m) 
\otimes_{\O_{Y_L}} \omega_{Y_L})$ can be represented by expressions $f 
\cdot dl_1 \wedge \cdots \wedge dl_d$, where $l_1, \ldots, l_d$ is a 
basis for $L^\vee$, and $f \in \Sym L^\vee$. Thus, we have embeddings
\begin{align*}
    H^0(Y_L, \O(m, \ldots, m) \otimes_{\O_{Y_L}} \omega_{Y_L}) &\subseteq 
    \det(L^\vee) \otimes_{\C} \Sym L^\vee,\quad \text{equivalently}\\
    \det(L) \otimes_{\C} H^0(Y_L, \O(m, \ldots, m) \otimes_{\O_{Y_L}} 
    \omega_{Y_L}) &\subseteq \Sym L^\vee.
\end{align*}
\end{remark}

We now recall a geometric interpretation of the inverse system 
$J_{L,-1}^\perp$ proved in \cite[Theorem 1.14 and Proposition 4.4]{CP}.

\begin{prop}\label{internal-gens}
 As graded subspaces of $\Sym L^\vee$,
    \[
        \det(L) \otimes_{\C} H^0(Y_L, \O_{Y_L}(1, \ldots, 1) \otimes_{\O_{Y_L}} \omega_{Y_L}) 
        = (J_{L,-1})^\perp.
    \]
\end{prop}

\begin{defn}\label{def:section-ring}
We define the following section ring as well as a module over it.
\begin{align*}
  R(Y_L, \O(1,\ldots,1)) &:= \bigoplus_{m \geq 0} H^0(Y_L, 
  \O_{Y_L}(m,\ldots,m)) \cdot x_0^m,\ \text{and}\\
  \widetilde{R}(Y_L, \O(1,\ldots,1)) &:= \bigoplus_{m \geq 0} H^0(Y_L, 
  \O_{Y_L}(m,\ldots,m) \otimes \omega_{Y_L}) \cdot x_0^m,
\end{align*}
where $R(Y_L, \O(1, \ldots, 1))$ is a subalgebra of $\Sym L^\vee 
\otimes \C[x_0]$ by \Cref{sections-subspace}, and likewise $\widetilde{R}(Y_L, 
\O(1, \ldots, 1))$ is an $R(Y_L, \O(1, \ldots, 1))$-submodule of $\Sym L^\vee \otimes_{\C} \C[x_0] 
\otimes_{\C} \det(L^\vee)$ by \Cref{forms-subspace}.
\end{defn}

We continue to explore the consequences of \Cref{multiplicity-free-thm} 
in the following Corollary, which will be engine behind the proof of 
\Cref{thrm:CM-intro}.

\begin{corollary}\label{Y-ring-and-module}
  The ring $R(Y_L,\O(1,\ldots,1))$ is the homogeneous coordinate ring of $Y_L$ under the Segre embedding $Y_L\subseteq \P^{2^n-1}$. It is Noetherian and Cohen-Macaulay, and 
  its canonical module is $\widetilde{R}(Y_L, \O(1,\ldots,1))$.
\end{corollary}
\begin{proof}
   The projective normality statement of \Cref{multiplicity-free-thm} means that $R(Y_L, \O(1,\ldots,1))$ is a 
    quotient of $R((\P^1)^n, \O(1, \ldots, 1))$. The latter is the homogeneous coordinate ring of the Segre embedding $(\P^1)^n \subseteq \P^{2^n - 1}$, and therefore both rings are Noetherian and $R(Y_L, \O(1,\ldots,1))$ is the homogeneous coordinate ring of $Y_L\subseteq \P^{2^n-1}$. 
    The fact that $R(Y_L, \O(1,\ldots,1))$ is Cohen-Macaulay is the content of the 
  statement that $Y_L$ is arithmetically Cohen-Macaulay. Finally, the 
  statement that $\widetilde{R}(Y_L, \O(1,\ldots,1))$ is the canonical 
  module is a general fact about projectively normal varieties, for 
  which we refer to the discussion in \cite[Section 
  1.1]{smithFujitasFreenessConjecture1997}.
\end{proof}

Recall (\Cref{dilation-def}) that the $m$-thickening of $L\subseteq 
\C^n$ is the image $L(m)$ of $L$ under the diagonal embedding $\C^n 
\subseteq (\C^n)^m$. Write $x_{ij}$ for the coordinates of $\C^{m \times 
n}$.

\begin{lemma}\label{dilation-lemma}
    As graded subspaces of $\Sym L^\vee$ and 
    $\det(L^\vee) \otimes_{\C} \Sym L^\vee$ respectively, we have
  \begin{align*}
    H^0(Y_{L(m)}, \O(1, \ldots, 1)) &= H^0(Y_L, 
    \O(m,\ldots,m)),\\
    H^0(Y_{L(m)}, \O(1, \ldots, 1) \otimes \omega_{Y_{L(m)}}) &= 
    H^0(Y_L, \O(m,\ldots,m)\otimes \omega_{Y_L}).
  \end{align*}
\end{lemma}

\begin{proof}
    The point is that the pullback of $\mathcal{O}_{(\P^1)^m}(1, \ldots, 
    1)$ over the diagonal embedding $\P^1 \hookrightarrow (\P^1)^m$ is 
    $\mathcal{O}_{\P^1}(m)$. Taking products, we get that the pullback 
    of $\mathcal{O}_{(\P^1)^{nm}}(1, \ldots, 1)$ over the diagonal 
    embedding in \Cref{dilation-def} is 
    $\mathcal{O}_{(\P^1)^n}(m,\ldots,m)$. The lemma now follows by the projection 
    formula, and the fact that $\omega_{Y_L} = \omega_{Y_{L(m)}}$ since 
    $Y_L \cong Y_{L(m)}$ as varieties.
\end{proof}

In the following proposition, we specialize from arbitrary $L \subseteq \C^n$ to the case of a unimodular zonotope.
\begin{prop}\label{section-ring-is-harmonic-algebra}
    Suppose that $Z = A \cdot [0,1]^n \subseteq \R^d$ is a unimodular 
    zonotope, and set $L := \rowspace(A) \otimes \C \subseteq \C^n$.
    Then there is an isomorphism of bigraded algebras and an isomorphism 
    of bigraded $\H_Z$-modules 
  \[
      \H_Z \cong R(Y_L, \O(1, \ldots, 1)),\ \ \ \text{and}\ \ \
    \widetilde{\H}_Z[0,d] \cong \widetilde{R}(Y_L, \O(1, \ldots, 1)), 
    \ \ \ \widetilde{\H}_Z[0,d] := \bigoplus_{i,j \in \Z} 
      \left(\widetilde{\H}_Z\right)_{i,j+d}.
  \]
\end{prop}
\begin{proof} 
    Let us start by proving the first isomorphism. We have the following 
    chain of equalities of graded subspaces of $\Sym L^\vee$.

    \begin{align*}
        (\H_Z)_{(m,*)} &= V_{mZ} &&\text{by \Cref{defn:harm},}\\
                       &= (J_{mZ}^\text{ext})^\perp && \text{by 
                       \Cref{thrm:zon_orbit},}\\
                       &= \C\left\{ \prod_{(i,j) \in S} l_{ij} 
                           \,\middle\vert\, S 
    \subseteq [m] \times [n] \right\},\quad l_{ij} := x_{ij} \vert_L, && \text{by \Cref{external-gens},}\\
                                      &= \text{Im}\left(
                                      H^0((\P^1)^{mn},\O(1,\ldots,1)) 
                                      \to 
                                      H^0(Y_{L(m)},\O(1,\ldots,1))\right)\\
                                      &= H^0(Y_{L(m)},\O(1,\ldots,1)) && 
                                      \text{by \Cref{multiplicity-free-thm},}\\
                                      &= H^0(Y_{L}, \O(m, \ldots, m)) && 
                                      \text{by \Cref{dilation-lemma}.}\\
    \end{align*}
    Multiplying by $x_0^m$ and summing over $m \geq 
    0$, yields $\H_Z = R(Y_L, \O(1, \ldots, 1))$ as bigraded subalgebras of $\Sym 
    L^\vee \otimes \C[x_0]$.

    Now we turn our attention to the proof of the second isomorphism in 
    the theorem statement. It is the same as the proof of the first 
    isomorphism, except that the spanning set of \Cref{external-gens} is replaced 
    with the geometric interpretation of \Cref{internal-gens}. Again, we 
    prove a chain of equalities of graded subspaces of $\Sym L^\vee$.
    \begin{align*}
        (\widetilde{\H_Z})_{(m,*)} &= \widetilde{V_{mZ}} &&\text{by \Cref{defn:harm},}\\
                       &= (J_{mZ}^\text{int})^\perp && \text{by 
                       \Cref{thrm:zon_orbit},}\\
                       &=  \det(L) \otimes_{\C} H^0(Y_{L(m)}, \O(1, \ldots, 
        1) \otimes \omega_{Y_{L(m)}}) && \text{by \Cref{internal-gens},}\\
                                      &= \det(L) \otimes_{\C} H^0(Y_{L}, \O(m, \ldots, m) 
        \otimes \omega_{Y_L}) && 
                                      \text{by \Cref{dilation-lemma}.}
    \end{align*}
    Summing over $m \geq 0$, yields $\widetilde{\H_Z} \cong \det(L) 
    \otimes_{\C} \widetilde{R}(Y_L, \O(1, \ldots, 1))$ as $\H_Z$-submodules of 
    $\Sym L^\vee \otimes \C[x_0]$. 
    Since the $\C^*$-action on $L$ is by inverse scaling, $\det(L)$ has 
    weight $-d$. Therefore the bihomogeneous graded component
    \[
        \det(L) \otimes_\C H^0(Y_L, \O(m, \ldots, m)\otimes 
        \omega_{Y_L})_i \otimes \C\cdot x_0^m \subseteq \det(L) 
        \otimes_\C \widetilde{R}(Y_L, 
        \O(1, \ldots, 1))
    \]
    has bidegree $(m, i-d)$, accounting for the degree shift.
\end{proof}
With this in place, we can now give proofs of \Cref{thrm:geometry} and \Cref{thrm:CM-intro}.
\begin{proof}[Proof of \Cref{thrm:geometry} and \Cref{thrm:CM-intro}]
\Cref{section-ring-is-harmonic-algebra} identifies the section ring of the line bundle $\O_{Y_L}(1, \ldots, 1)$ with 
    the harmonic algebra $\H_Z$ and the canonical module of the section 
    ring with the shifted interior ideal $\widetilde{\mathcal{H}_Z}$. Both claims now follow from \Cref{Y-ring-and-module}.
\end{proof}

We now turn our attention to proving 
\Cref{thrm:pres-intro,cor:harmonic_pres}. We start by 
proving the commutativity of a diagram which will let us write 
$\mathcal{H}_Z$ as a homogeneous quotient of $\mathbb{C}[z_S\vert 
S\subseteq [n]]$.

\begin{prop}\label{coordinate-ring-diagram}
    We have the following commutative diagram, where $i$ ranges over 
    $[n]$, $S$ ranges over subsets of $[n]$, $\Phi(z_S) = 
    \prod_{i \in S}x_i \prod_{i \not\in S}y_i$, and $\widetilde{I_L}$ is 
    the multihomogeneous circuit ideal of \Cref{circuit-equations}. 
    \[\begin{tikzcd}
    \C[x_i,y_i] \arrow[r, "\cong"] & 
    \bigoplus_{u \in 
\Z_{\geq 0}^n} H^0((\P^1)^n, \O(u)) \arrow[r, two heads] & \bigoplus_{u \in 
\Z_{\geq 0}^n} H^0(Y_L, \O(u)) \arrow[r, "\cong"] & \C[x_i,y_i]/\widetilde{I_L}\\
\C[z_S]\arrow[u, "\Phi"]\arrow[r,two heads] & R((\P^1)^n, 
\O(1, \ldots, 1)) \arrow[u, hook]\arrow[r,two heads] & R(Y_L, \O(1, 
\ldots, 1)) \arrow[u,hook] \arrow[r,"\cong"] & 
\C[z_S]/\Phi^{-1}(\widetilde{I_L})\arrow[u, hook]
\end{tikzcd}\]
\end{prop}
\begin{proof}
    The top left horizontal arrow is a standard identification, where 
    the sections of $\O_{(\P^1)^n}(u)$ are identified with polynomials 
    of $\Z_{\geq 0}^n$-multidegree $u$ in $\C[x_i,y_i]$.
    The top middle arrow is surjective by \Cref{multiplicity-free-thm}. 
    The multihomogeneous defining equations $\widetilde{I_L}$ are by 
    definition the kernel of the map $\C[x_i,y_i] \to \bigoplus_{u \in 
    \Z_{\geq 0}^n}H^0(Y_L, \O(u))$, so the top right arrow is an 
    isomorphism. The ring map $\Phi$ is surjective onto the degree $(m, 
    \ldots, m)$ component of $\C[x_i,y_i]$ for all $m$, so the bottom 
    left arrow is surjective. The second, third, and fourth vertical arrows all 
    denote the inclusion of the diagonal $\Z_{\geq 0}^n$-graded 
    piece, and are thus injective.
\end{proof}

\begin{corollary}\label{this-is-segre}
  Let $\varphi_L$ be the following composition in 
\Cref{coordinate-ring-diagram}.
\[
\begin{tikzcd}
\mathbb{C}[z_S\vert S\subseteq [n]] \arrow[r, "\Phi"] & \mathbb{C}[x_1,y_1,\ldots,x_n,y_n] \arrow[r] & \frac{\mathbb{C}[x_1,y_1,\ldots,x_n,y_n]}{\widetilde I_L}
\end{tikzcd}\]
The bigraded homogeneous coordinate ring of $Y_L\subseteq \P^{2^n-1}$ is isomorphic to  $\mathbb{C}[z_S: 
S\subseteq [n]]/\ker \varphi_L$ where the generator $z_S$ has $t$-degree 
$1$ and $q$-degree $|S|$. 
\end{corollary}

Our goal is to give an explicit presentation of the ideal $\ker \varphi_L$.

\begin{defn} 
For a circuit $C$ and a subset $A\subseteq [n]\setminus C$, define the linear function $f_C^A(z)$ as
\[ f_C^A(z) = \sum_{i\in C} 
\alpha_{C,i}z_{A\cup \{i\}}\]
where the $\alpha_{C,i}$ are the numbers defined in \Cref{circuit-equations}. Define the ideal $I_L^\text{SE} \subseteq \mathbb{C}[z_S\vert  S\subseteq [n]]$ as
\[
    I_L^{\text{SE}} = \left\langle z_Sz_T-z_{S\cup T}z_{S\cap T}, f_C^A(z) \,\vert\,
S,T\subseteq [n], C \text{ a circuit and } A\subset [n]\setminus 
C\right\rangle.
\]
Write $R_L \coloneqq \C[z_S \mid S \subseteq [n]]/I_L^{\text{SE}}$ for the 
associated quotient algebra.
\end{defn}

We can now restate \Cref{thrm:pres-intro} as the 
following theorem.

\begin{thrm}\label{prop:harmonic-pres}
    We have that $\ker \varphi_L = I_L^{\text{SE}}$ and hence that  $\mathcal{H}_Z \simeq R_L$. 
  \end{thrm}
  
 As the following lemma shows, it is straightforward to check that $I_L^\text{SE} \subseteq \ker \varphi_L$.
\begin{lemma}\label{lem:harmonic-containment}
    The ideal $I_L^\text{SE}$ is contained in the kernel of $\varphi_L$.
\end{lemma}
\begin{proof}
   For any $S,T\subset [n]$, $z_Sz_T-z_{S\cup T}z_{S\cap T}$ is 
   contained in the kernel of $\varphi$ since it is contained in the 
   kernel of $\Phi$. Now for a circuit $C$ and subset 
   $A\subseteq [n]\setminus C$,
    \[\Phi(f_C^A(z))= \sum_{i\in C} \alpha_{C,i} \prod_{i\in A\cup \{i\}}x_i \prod_{j\not\in A\cup \{i\}}y_j = \widetilde{f}_C \prod_{i\in A}x_i \prod_{j \in ([n]\setminus C) \setminus A} y_j \]
    which is contained in $\widetilde{I}_L$ and hence in the kernel of 
    $\varphi_L$.
\end{proof}

For the rest of this section, we will ignore the $q$-grading. 
    \Cref{lem:harmonic-containment} shows that $R_L$ surjects onto 
    $R(Y_L, \O(1, \ldots, 1))$, and so the following lemma gives us a 
    lower bound on $I_L^{SE}$.
    \begin{lemma}\label{lem:ISE-lower-bound}
        We have that $\Hilb(R(Y_L, \O(1, \ldots 1)), t) = \sum_{t \geq 0} m^d 
        T_M(\frac{m+1}{m}, 1) t^m.$
    \end{lemma}
    \begin{proof} We have the following chain of equalities. The 
        proof of the third equality is the same as in the proof of 
        \Cref{section-ring-is-harmonic-algebra}.
        \begin{align*}
            \dim R(Y_L, \O(1, \ldots&, 1))_m\\
            &= \dim H^0(Y_L, \O(m, 
            \ldots, m)), && \text{by \Cref{def:section-ring}},\\
                         &= \dim H^0(Y_{L(m)}, \O(1, \ldots, 1)), && 
                         \text{by \Cref{dilation-lemma}},\\
                         &= \dim (J_{L(m)}^{\text{ext}})^\perp && 
                         \text{by 
                         \Cref{external-gens,multiplicity-free-thm}},\\
                         &= \dim R_{L(m)}^\text{ext}, \\
                         &= T_{M(m)}(2,1), && \text{by 
                         \Cref{zonotopal-hilbert}},\\
                         &= m^d T_{M}(\frac{m+1}{m}, 1), && \text{by 
                         \Cref{dilated-zonotope-vs-matroid}}.\qedhere
        \end{align*}
    \end{proof}
    \Cref{lem:ISE-lower-bound} shows that in order to prove 
    \Cref{prop:harmonic-pres}, we just need to prove that  
\[\dim(R_L)_m \leq  m^d T_M\left(\frac{m+1}{m},1\right),\quad \text{for 
all $m \geq 0$}.  \]

We will now work towards reducing this dimension check to the case where $m=1$. Before doing so, we make the following observation.

Recall from \Cref{zero-one-matrix-notation} that we consider a subset $S 
\subseteq [m] \times [n]$ as a $0/1$ matrix. We use the notation $S_{i,*}$ and $S_{*,j}$ to indicate the $i$th row and $j$th column of $S$, thought of as subsets of $[n]$ and $[m]$, respectively. We will also refer to the row support and column support of $S$. These are the sets of indices $i\in [m]$ and $j\in [n]$ such that $|S_{i,*}|>0$ and $|S_{*,j}|>0$, respectively. 

\begin{lemma}\label{lem:multiset-relations}
    Let $S$ and $T$ be subsets of $[m] \times [n]$ such that, for all $j\in[n]$, the number of ones in the $j$th column of $S$ is equal to the number of ones in the $j$th column of $T$. 
    The  binomial 
    \[
        \prod_{i=1}^m z_{S_{i,*}} - \prod_{i=1}^m z_{T_{i,*}}
    \]
    is contained in $I_L^{\text{SE}}$.
\end{lemma}
\begin{proof}
    To prove the claim, we will show that the two terms of the binomial are equal in $R_L$. We proceed by induction on the largest number of ones in a column of $S$ (which is equal to the largest number of ones in a column of $T$). In the base case, both terms are equal to $z_\emptyset^m$ and the claim holds. For the inductive case, we can repeatedly use the binomial relation of $I_L^{\text{SE}}$ to add ones to the last row of $S$ and remove ones from the higher rows. In the end, we will construct a subset $S'\subseteq [m]\times [n]$ such that $S'_{m,*}= \colsupp(S)$ and 
    \[\prod_{i=1}^m z_{S_{i,*}} =\prod_{i=1}^m z_{S'_{i,*}} = z_{\colsupp(S)}\prod_{i=1}^{m-1} z_{S'_{i,*}}. \]
    We can do a similar construction to find a $T'\subseteq [m]\times [n]$ such that 
    \[\prod_{i=1}^m z_{T_{i,*}} =\prod_{i=1}^m z_{T'_{i,*}} = z_{\colsupp(T)}\prod_{i=1}^{m-1} z_{T'_{i,*}}. \]
    The submatrix of $S'$ formed by the first $m-1$ rows has exactly $|S_{*,i}|-1$ ones in its $i$th column. Similarly, the submatrix of $T'$ formed by the first $m-1$ rows of $T'$ has exactly $|T_{*,i}|-1$ ones in its $i$th column. As $|S_{*,i}|=|T_{*,i}|$, it follows from our induction hypothesis that 
    \[\prod_{i=1}^{m-1} z_{S'_{i,*}}=\prod_{i=1}^{m-1} z_{T'_{i,*}}. \]
    Furthermore, our assumptions imply that $\colsupp(T)=\colsupp(S)$ so
    \[ \prod_{i=1}^m z_{S_{i,*}} = z_{\colsupp(S)}\prod_{i=1}^{m-1} z_{S'_{i,*}}=z_{\colsupp(T)}\prod_{i=1}^{m-1} z_{T'_{i,*}}=\prod_{i=1}^m z_{T_{i,*}}.\qedhere 
    \]

\end{proof}

Recall from \Cref{dilation-def} that the ground set of the 
$m$-thickening $L(m)$ is $[m] \times [n]$.
\begin{lemma}\label{lem:reduce-to-linear}
There is a surjective linear map
\[\gamma: (R_{L(m)})_1 \twoheadrightarrow (R_L)_m,\quad \gamma (z_S) = 
\prod_{i=1}^m z_{S_{i,*}}. \]
\end{lemma}
\begin{proof}
First, we observe that the map
\[
    \tilde \gamma: \mathbb{C}[z_S\vert S\subseteq [m]\times [n]]_1 \to 
\mathbb{C}[z_S\vert S\subseteq [n]]_m,\quad \tilde \gamma (z_S) = 
\prod_{i=1}^m z_{S_{i,*}}
\]
 is surjective. Therefore to prove the lemma, we just need to check that $\gamma$ is well defined. We do this by checking that $\tilde \gamma ( I_{L(m)}^{\text{SE}})_1 \subseteq (I_L^{\text{SE}})_m$. 
By \Cref{dilation-def}, the space $( I_{L(m)}^{\text{SE}})_1$ is spanned by the set of all linear functionals 
\[
    f_C^A(z) = \sum_{(i,j) \in C}\alpha_{\operatorname{colsupp}(C), 
    j}z_{A \cup (i,j)}\quad \text{and}\quad f_{k,\ell}^D = z_{D \cup (k,\ell)} 
    - z_{D \cup (k+1,\ell)},
\]
where 
\begin{itemize}
    \item $C\subseteq[m]\times[n]$ is such that $|C_{*,j}|\leq 1$ for all $j\in [n]$ and $\colsupp(C)$ is a circuit of $M$,
    \item $A\subseteq [m]\times [n]$ is disjoint from $C$,
    \item $(k,\ell)\in [m-1]\times [n]$, and
    \item $D\subseteq [m]\times [n]$ is disjoint from $\{(k,\ell), (k+1,\ell)\}$.
\end{itemize} 

The sets $D\cup (k,\ell)$ and $D\cup (k+1,\ell)$ have equal number of ones in each of their columns. Therefore, by \Cref{lem:multiset-relations},
\[\tilde\gamma(f_{k,\ell}^D) = \prod_{i=1}^m z_{\left (D\cup (k,\ell) \right)_{i,*}} - \prod_{i=1}^m z_{\left (D\cup (k,\ell+1) \right)_{i,*}} \in I_L^{\text{SE}}. \]
Now consider the image of $f_C^A(z)$. We claim 
that $\tilde\gamma(f_C^A(z))$ is divisible by the linear functional 
$f_{\operatorname{colsupp}(C)}^{\operatorname{colsupp}(A) \setminus 
\operatorname{colsupp}(C)}(z) \in I_L^{\text{SE}}$, which will 
prove our claim. By definition,
\begin{equation}\label{eq:to_transform}
\tilde\gamma(f_C^A(z)) = \sum_{(i,j) \in C} 
\alpha_{\operatorname{colsupp}(C),j} \prod_{k=1}^m z_{\left(A\cup (i,j)\right)_{k,*}}\,.
\end{equation}

We now transform the monomials appearing in the right hand side of \Cref{eq:to_transform} 
using \Cref{lem:multiset-relations}. Fix $(i,j) \in C$, and consider the 
monomial of the sum corresponding to the subset 
$A \cup (i,j)$. 
Notice that $A \cup (i,j)$ has at least one zero in each column of 
$\operatorname{colsupp}(C) \setminus j$, since $C \cap A = \emptyset$.
This observation ensures that there exists a set $A'\subseteq [m]\times [n]$ such that $|(A\cup (i,j))_{*,k}| = |A'_{*,k}|$ for all $k\in [n]$, and the $i$th row of $A'$ is equal to $A'_{i,*}= \left(\colsupp(A)\setminus \colsupp(C)\right) \cup \{j\}$. The set $A'$ can constructed from $A\cup (i,j)$ by first moving ones into its $i$th row so that the $i$th row is equal to $\colsupp(A)\cup \{j\}$ and then moving ones in the columns of $\colsupp(C)\setminus \{j\}$ off of this row and into the gap guaranteed by the disjointedness of $C$ and $A$.
See \Cref{circuit-equation-move} for an illustration. By \Cref{lem:multiset-relations},
\[ \prod_{k=1}^m z_{\left(A\cup (i,j)\right)_{k,*}} = \prod_{k=1}^m z_{A'_{k,*}} = z_{\left(\colsupp(A)\setminus \colsupp(C)\right) \cup \{j\}} \prod_{k\in [m]\setminus i} z_{A'_{k,*}}.\]
Let $\tilde A$ be the submatrix of $A'$ formed by removing its $i$th row. By construction,
\[|\tilde A_{*,k}| = \begin{cases}
    |A_{*,k}|-1,& k \in \colsupp(A)\setminus \colsupp(C)\\
    |A_{*,k}|,& \text{ else.}
\end{cases}  \]
Note that the number of ones in each column of $\tilde A$ does not depend on our choice of $(i,j)\in C$. Thus by \Cref{lem:multiset-relations}, the monomial $\prod_{k\in [m]\setminus i} z_{A'_{k,*}}$, up to the relations of $I_L^{\text{SE}}$, does not depend on our choice of $(i,j)$. Let $g\in \mathbb{C}[z_S\vert S\subseteq [n]]$ be a representative of this monomial. Our computations now show that
\[\tilde\gamma(f_C^A(z))=g \sum_{(i,j) \in C} \alpha_{\operatorname{colsupp}(C),j} z_{\left(\colsupp(A)\setminus \colsupp(C)\right) \cup \{j\}} = g f_{\operatorname{colsupp}(C)}^{\operatorname{colsupp}(A) \setminus 
\operatorname{colsupp}(C)}(z), \]
up to the relations in $I_L^{\text{SE}}$.
\end{proof}

\begin{figure}[h!]
\[
A \cup (2,3) = \begin{pmatrix}
    \boxed{0} & \boxed{0} & 0 & 1 & 1 \\
    0 & 1 & \boxed{1} & \boxed{0} & 0 \\
    0 & 0 & 0 & 1 & 1
\end{pmatrix},\quad 
A'  = \begin{pmatrix}
    \boxed{0} & \boxed{\red{1}} & 0 & 1 & 1 \\
    0 & \red{0} & \boxed{1} & \boxed{0} & \red{1} \\
    0 & 0 & 0 & 1 & \red{0}
\end{pmatrix}
\begin{array}{c}
\\[-1.2ex]
\longleftarrow i=2\\
\\[-1.2ex]
\end{array}
\]
\caption{
    An example from the proof of \Cref{lem:reduce-to-linear}, where $m = 
    3$, $n = 5$, $(i,j)=(2,3)$, and $C$ is indicated by the boxed entries. The red entries highlight which elements of the columns we've altered to construct $A'$. The matrix $A'$ has the property that its second row is equal to the set $(\colsupp(A)\setminus\colsupp(C)) \cup \{3\} = \{3,5\}$. 
}\label{circuit-equation-move}
\end{figure}

Lemma \ref{lem:reduce-to-linear} lets us reduce the proof of Proposition \ref{prop:harmonic-pres} to analyzing the degree one case. Our analysis of the degree one case follows a deletion-contraction argument similar in spirit to the proof of \cite[Proposition 4.4]{AP10}.

\begin{lemma}\label{lem:dim1-case}
The dimension of $(R_L)_1$ is $T_M(2,1)$, the number of independent sets of $M$.
\end{lemma}
\begin{proof}
    We prove the statement by induction on the number of elements $i\in [n]$ which are neither loops nor coloops. In the base case, every element is either a loop or coloop and $(I^{\text{SE}}_L)_1$ is spanned by the variables
    \[\{ z_S : S \text{ contains a loop}\}. \]
    In this case, $(R_L)_1$ has basis $\{z_{T}: T \text{ does not contain a loop}\}$. As every element which is not a loop is instead a coloop, this means that $\dim(R_L)_1$ is equal to the number of independent sets of $M$.

    Now suppose that $i$ is an element of $[n]$ which is neither a 
    loop nor a coloop. By our inductive hypothesis, we can assume 
    that the claim is true for $L/i$ and $L\setminus i$. Consider the 
    short exact sequence
\[\begin{tikzcd}
0 \arrow[r] & {\mathbb{C}[z_S\ordinarycolon S\subseteq [n]\setminus \{i\}]_1} \arrow[r, "\iota"] & {\mathbb{C}[z_S\ordinarycolon S\subseteq [n]]_1} \arrow[r, "\pi"] & {\mathbb{C}[z_S\ordinarycolon S\subseteq [n]\setminus \{i\}]_1} \arrow[r] & 0
\end{tikzcd}\]
\[\pi(z_S):=\begin{cases}
    z_S, &i\not\in S\\
    0, &\text{else}
\end{cases} \quad\text{and}\quad \iota(z_S) := z_{S\cup i}.\]
We claim that this sequence induces the following diagram whose rows are exact:
\[\begin{tikzcd}
0 \arrow[r] & \left(I_{L\setminus i}^{\text{SE}}\right)_1 \arrow[d, hook] \arrow[r]                                                  & \left(I_{L}^{\text{SE}}\right)_1 \arrow[d, hook] \arrow[r]                                            & \left(I_{L/ i}^{\text{SE}}\right)_1 \arrow[d, hook] \arrow[r]                                                 & 0 \\
0 \arrow[r] & {\mathbb{C}[z_S\ordinarycolon S\subseteq [n]\setminus \{i\}]_1} \arrow[r, "\iota"] \arrow[d, two heads] & {\mathbb{C}[z_S\ordinarycolon S\subseteq [n]]_1} \arrow[r, "\pi"] \arrow[d, two heads] & {\mathbb{C}[z_S\ordinarycolon S\subseteq [n]\setminus \{i\}]_1} \arrow[r] \arrow[d, two heads] & 0 \\
0 \arrow[r] & \left(R_{L\setminus i}\right)_1 \arrow[r]                                                                              & \left(R_L\right)_1 \arrow[r]                                                                          & \left(R_{L/i}\right)_1 \arrow[r]                                                                              & 0
\end{tikzcd}\, .\]
Along the way to proving that this diagram is exact, we will obtain our 
claim. We first verify that $\iota( I_{L\setminus i}^{\text{SE}})_1 
\subseteq \ker(\pi)\cap (I_L^{\text{SE}})_1$ and $\pi(I_L^{\text{SE}})_1 
\subseteq (I_{L/i}^{\text{SE}})_1$ so that our diagram is well defined. 
The subspace $(I_L^{\text{SE}})_1$ is spanned by the linear functions 
$f_C^A(z)$ for $C$ a circuit of $L$ and $A\subseteq [n]\setminus C$. 
The circuits of the deletion $M\setminus i$ are those 
circuits of $M$ that do not contain $i$, and the circuits of the 
contration $M/i$ are minimal nonempty sets $C\setminus i$ where $C$ is a 
circuit of $M$ \cite[Proposision 3.1.10 and Claim 
3.1.13]{oxleyMatroidTheory2011}. Therefore, the subspaces $(I_{L\setminus 
i}^{\text{SE}})_1$ and $(I_{L/i}^{\text{SE}})_1$ are spanned by the sets 
of linear functions
\[\{f^A_C(z): i\not\in C, A\subseteq ([n]\setminus i) \setminus C\} \quad\text{and}\quad \{f^A_{C\setminus i}(z): A\subseteq ([n]\setminus i) \setminus (C\setminus i)\} ,\]
respectively. The map $\iota$ sends a linear function $f^A_C(z)$ to $f^{A\cup i}_C(z)$. As $i\not\in C$, $f^{A\cup i}_C(z)\in (I_L^{\text{SE}})_1$. As every variable in the support of $f^{A\cup i}_C(z)$ is indexed by a set containing $i$, $f^{A\cup i}_C(z) \in \ker(\pi)$. This proves our first claimed inclusion.

For the second inclusion, note that $\pi(f_C^A(z))$ is zero when $i\in 
A$. If $i\not\in A$, then $\pi(f_C^A(z)) = f_{C\setminus i}^A(z)$. In 
fact, this computation shows that $(I_L^{\text{SE}})_1$ surjects onto 
$(I_{L/i}^{\text{SE}})_1$. Therefore, we have shown that the top row of 
our diagram is exact, except possibly at $(I_L^\text{SE})_1$. Here, we 
will use our induction hypothesis. Our computations thus far have shown 
that $\dim(I_L^{\text{SE}})_1 \geq \dim(I_{L\setminus i}^{\text{SE}})_1 
+ \dim(I_{L\setminus i}^{\text{SE}})_1$. The top row of our diagram will 
be exact if this inequality is an equality. We can also compute that
\begin{align*}
    \dim(R_L)_1 &= \mathbb{C}[z_S\ordinarycolon S\subseteq [n]]_1 - 
    \dim(I_L^{\text{SE}})_1\\
                                                                 &\leq 2\cdot 2^{n-1} - \dim(I_{L\setminus 
i}^{\text{SE}})_1 - \dim(I_{L\setminus i}^{\text{SE}})_1\\ 
&= 
\dim(R_{L\setminus i})_1 + \dim(R_{L/i})_1.
\end{align*}
Our induction hypothesis tells us that 
\[\dim(R_{L\setminus i})_1 + \dim(R_{L/i})_1 = T_{M\setminus i}(2,1) + T_{M/i}(2,1) = T_{M}(2,1). \]
Thus we have shown that $\dim (R_L)_1 \leq T_M(2,1)$ and equality is only obtained if the top row of our diagram is exact. However, by \Cref{lem:harmonic-containment,lem:ISE-lower-bound}, we already know that $\dim(R_L)_1 \geq T_M(2,1)$. Thus $\dim(R_L)_1 = T_M(2,1)$ and the top row of our diagram is exact. Exactness of the entire diagram follows from an application of the snake lemma.
\end{proof}

\begin{remark}
    The vector space $(R_L)_1$ is dual to the subspace $L^{\text{SE}} \subseteq \mathbb{C}^{2^{[n]}}$ cut out by the linear functions $f_C^A(v)$. \Cref{lem:dim1-case} tells us that the rank of the matroid $M^{\text{SE}}$ of $L^\text{SE}$ is equal to the number of independent sets of $M$. It would be interesting to describe the matroid $M^{\text{SE}}$. The collection
    \[\{\{i \cup A: i\in C\}: \text{$C$ a circuit of $M$}, A\subseteq [n]\setminus C\} \]
    is contained in the set of circuits of $M^{\text{SE}}$ but this containment is often strict. If one could describe the set of circuits of $M^\text{SE}$, this could give insight into the bases of $M^\text{SE}$. Explicitly constructing a basis of $M^{\text{SE}}$ would give a satisfying alternative proof of \Cref{lem:dim1-case}.
\end{remark}

We are now able to prove \Cref{prop:harmonic-pres}.

\begin{proof}[Proof of \Cref{prop:harmonic-pres}]
As noted earlier, \Cref{lem:harmonic-containment} reduces our claim to 
checking that $\dim(R_L)_m \leq m^d T_M(\frac{m+1}{m},1)$ for all $m\geq 
0$. By \Cref{lem:reduce-to-linear}, \Cref{lem:dim1-case} and \Cref{dilated-zonotope-vs-matroid},
\[\dim(R_L)_m \leq \dim(R_{L(m)})_1 = T_{M(m)}(2,1) = m^d 
T_M(\frac{m+1}{m},1).\qedhere\]
\end{proof}

\section{Gorenstein Harmonic Algebras}\label{Gorenstein-section}
In this section, we characterize the unimodular zonotopes $Z = A\cdot [0,1]^n \subseteq \R^d$ whose harmonic algebras $\mathcal{H}_Z$ are Gorenstein in terms of the matroid $M$ of $A$. We then show that if $\mathcal{H}_Z$ is Gorenstein, the numerator of the graded Ehrhart series of $Z$ displays a remarkable palindromic symmetry.

Recall that, by \Cref{section-ring-is-harmonic-algebra}, we know $\widetilde {\H_Z}[0,d]\simeq \Omega\mathcal{H}_Z$. Thus, $\mathcal{H}_Z$ is Gorenstein if and only if $\mathcal{H}_Z$ is isomorphic to $\widetilde{\mathcal{H}_Z}$ as graded $\mathcal{H}_Z$-modules (possibly with a degree shift). This observation is the starting point of our proof of \Cref{thrm:gor}. We now collect a series of lemmas to analyze the interior ideal $\widetilde{\mathcal{H}_Z}$ of $Z$.

\begin{lemma}\label{lem:domain}
The harmonic algebra $\mathcal{H}_Z$ is an integral domain. 
\end{lemma}
\begin{proof}
  By \Cref{cor:harmonic_pres}, $\mathcal{H}_Z$ is the homogeneous coordinate ring of $Y_L\subseteq \P^{2^n-1}$ under the Segre embedding. The variety $Y_L$ is irreducible, hence $\mathcal{H}_Z$ contains no zero divisors.
\end{proof}

\begin{lemma}\label{lem:sum-tutte}
  Let $M$ be a matroid with $k$ connected components such that each connected component is also a circuit of $M$. Let $d_1,d_2,\ldots , d_k$ be the ranks of these respective components.
  The Tutte polynomial of $M$ is equal to
  \[T_M(x,y) = \prod_{i=1}^k (y+x+x^2+\ldots + x^{d_i}). \]
\end{lemma}
\begin{proof}
The Tutte polynomial of a matroid is equal to the product of the Tutte polynomials of its connected components. Therefore, it suffices to prove that $T_{U_{n-1,n}}(x,y) = y+x+x^2+\ldots + x^{n-1}$. This, in turn, follows from a deletion-contraction argument.
\end{proof}

In the following lemma, we use the presentation of $\mathcal{H}_Z$ given in \Cref{prop:harmonic-pres}.
\begin{lemma}\label{lem:empty_contained}
If $M$ contains no coloops, the term $z_\emptyset$ represents an element of the interior ideal $\widetilde{\mathcal{H}_Z}$. In any case, the term $z_{\emptyset}^2$ represents an element of the interior ideal $\widetilde{\mathcal{H}_Z}$.
\end{lemma}
\begin{proof}
  Under the map $\varphi_L$, $z_\emptyset$ is sent to the element of $\C[x_1,y_1,\ldots,x_n,y_n]/\widetilde{I}_L$ represented by $y_1y_2\cdots y_n$. This restricts to the function $1\in \Sym L^\vee$ on $L$. If $M$ contains no coloops, every cocircuit vector $v\in L$ has $|\text{Supp}(v)|>1$. Thus the homogeneous generators $v^{m(v)-1} \in \Sym L$ of $J_Z^{\text{int}}$ have positive degree. This ensures that $1\in (J_Z^{\text{int}})^\perp\subseteq \Sym L^\vee$. The claim now follows from \Cref{thrm:zon_orbit}.

  For any matroid $M$, the $2$-thickening of $M$ contains no coloops. To see this, note that if $(1,i) \sqcup S$ is a basis of $M(2)$ then $(2,i) \sqcup S$ is also a basis of $M(2)$. From the argument in the previous paragraph, it follows that $1\in J_{2 Z}^{\text{int}} \subseteq \Sym L(2)^\vee = \Sym L^\vee$. As $z_{\emptyset}$ restricts to $1\in \Sym L^\vee$, $z_\emptyset^2$ also restricts to $1\in \Sym L^\vee$. The claim now follows.
\end{proof}

Let $(\widetilde{\mathcal{H}_Z})_m$ denote the $t$-degree $m$ subspace of the interior ideal.
The following lemma will prove one direction of \Cref{thrm:gor}.
\begin{lemma}\label{lem:interior-dim}
  Let $m_0$ be the smallest index such that $\dim (\widetilde{\mathcal{H}_Z})_{m_0} >0$.
  \begin{itemize}
  \item If every connected component of $M$ is a circuit, then $m_0=1$ and $\dim (\widetilde{\mathcal{H}_Z})_1 =1$.
  \item If $M$ is Boolean, then $m_0=2$ and $\dim (\widetilde{\mathcal{H}_Z})_2 =1$.
  \item Otherwise, $\dim (\widetilde{\mathcal{H}_Z})_{m_0} >1$.
  \end{itemize}
\end{lemma}
\begin{proof}
  We prove all three items simultaneously.
  By definition, we have that $m_0\geq 1$. The dimension of $(\widetilde{\mathcal{H}_Z})_1$ is equal to $T_M(0,1)$. The value $T_M(0,1)$ is equal to the M\"obius invariant of $M^*$, the dual matroid of $M$. The M\"obius invariant of $M^*$ is equal to zero if and only if $M$ has a coloop.

  First, assume that $M$ does not have a coloop. By \cite[Exercise 7.8.7]{bjornerap}, the M\"obius invariant of $M^*$ is equal to one if and only if every connected component of $M$ is a circuit and is strictly larger than one otherwise. Thus our claim holds in the case where $M$ does not have a coloop.

  Now suppose $M$ has a coloop. In this case, $\dim(\widetilde{\mathcal{H}_Z})_1 = 0$. By the definition of the interior ideal (\Cref{defn:harm}) and \Cref{thrm:stanley_ehrhart},
  \begin{equation}\label{eq:second_dilate}
      \dim(\widetilde{\mathcal{H}_Z})_2 = \vert 2Z \cap \Z^d \vert = 2^d T_M(\frac{1}{2},1).
  \end{equation}
  As the Tutte polynomial has non-negative integer coefficients and is of $x$-degree at most $d$ (\Cref{prop:tutte_def}), the right hand side of \Cref{eq:second_dilate} is a non-empty sum of non-negative powers of $2$. Such a sum is always positive and is equal to $1$ if and only if $T_M(x,y) = x^d$. This occurs exactly when $M$ is the Boolean matroid.
\end{proof}
With these lemmas in place, we can now prove \Cref{thrm:gor}.
\begin{proof}[Proof of \Cref{thrm:gor}]
  First, suppose $M$ is not the Boolean matroid nor a direct sum of circuits. Let $m_0$ be the smallest index such that $(\widetilde{\mathcal{H}}_Z)_{m_0} \neq 0$. \Cref{lem:interior-dim} tells us that $\dim(\widetilde{\mathcal{H}}_Z)_{m_0} >1$. However $\dim(\mathcal{H}_Z)_0 =1$, so we cannot have a graded module isomorphism $\mathcal{H}_Z \simeq \widetilde{\mathcal{H}_Z}$, even with a degree shift.

 Let $M$ be a matroid with $k$ connected components such that each connected component is also a circuit of $M$. Let $d_1,d_2,\ldots , d_k$ be the ranks of these respective components.  In this case, we claim that $\widetilde{\mathcal{H}_Z}\simeq \langle z_{\emptyset}\rangle$. By \Cref{lem:empty_contained}, $\langle z_{\emptyset}\rangle \subseteq \widetilde{\mathcal{H}_Z}$. We now do a dimension check to show that $\langle z_{\emptyset}\rangle = \widetilde{\mathcal{H}_Z}$.

  As $\mathcal{H}_Z$ is an integral domain (\Cref{lem:domain}), $\langle z_{\emptyset} \rangle_{m+1} \simeq (\mathcal{H}_Z)_m$. Therefore, by \Cref{thrm:stanley_ehrhart}, 
  \[\dim \langle z_{\emptyset} \rangle_{m+1}  = \dim (\mathcal{H}_Z)_m=  m^dT_M(\frac{m+1}{m},1).\]
  We also know that
  \[\dim(\widetilde{\mathcal{H}_Z})_{m+1} = (m+1)^dT_M(\frac{m}{m+1},1). \]
  By \Cref{lem:sum-tutte},$ T_M(x,1)= \prod_{i=1}^k (1+x+x^2+\ldots+x^{d_i})$. Therefore,
  \begin{align*}
    m^dT_M(\frac{m+1}{m},1) &= \prod_{i=1}^km^{d_i}\left(1+\frac{m+1}{m}+\ldots + \left(\frac{m+1}{m}\right)^{d_i}\right)\\
                            &= \prod_{i=1}^k \left(m^{d_i}+m^{d_i-1}(m+1)+\ldots + m(m+1)^{d_i-1}+(m+1)^{d_i}\right)\\
                            &=\prod_{i=1}^k(m+1)^{d_i}\left(1+\frac{m}{m+1}+\ldots + \left(\frac{m}{m+1}\right)^{d_i}\right)\\
    &= (m+1)^dT_M(\frac{m}{m+1},1)\,.
  \end{align*}
  Thus, in this case, we've shown that
  \[\mathcal{H}_Z[1,d]\simeq \widetilde{\mathcal{H}_Z}[0,d] \simeq \Omega \mathcal{H}_Z \]
  and that $\mathcal{H}_Z$ is Gorenstein.

  If $M$ is the Boolean matroid, then we claim that $\widetilde{\mathcal{H}_Z} = \langle z_{\emptyset}^2 \rangle$. We proceed in a manner that is similar to the previous case. By \Cref{lem:empty_contained}, $\langle z_{\emptyset}^2\rangle$ is contained in the interior ideal. As $\mathcal{H}_Z$ is an integral domain (\Cref{lem:domain}), $\langle z_{\emptyset}^2 \rangle_{m+2} \simeq (\mathcal{H}_Z)_m$. Therefore, by \Cref{thrm:stanley_ehrhart}, 
  \[\dim \langle z_{\emptyset}^2 \rangle_{m+2}  = \dim (\mathcal{H}_Z)_m = m^dT_M(\frac{m+1}{m},1) = (m+1)^d.\]
  We also know that
  \[\dim(\widetilde{\mathcal{H}_Z})_{m+2} = (m+2)^dT_M(\frac{m+1}{m+2},1)= (m+1)^d.\]
  Thus, in this case, we've shown that
  \[\mathcal{H}_Z[2,d]\simeq \widetilde{\mathcal{H}_Z}[0,d] \simeq \Omega \mathcal{H}_Z \]
  and that $\mathcal{H}_Z$ is Gorenstein.
\end{proof}
Recall that \Cref{thrm:ehrhart_series} lets us write the graded Ehrhart series of $Z$ as a rational function
\[E_Z(t,q)= \frac{N_Z(t,q)}{\prod_{i=0}^n (1-tq^i)} \]
where $N_Z(t,q)\in \mathbb{Z}[t,q]$ is of $t$-degree at most $n$. Write $N_Z(t,q) = \sum_{k=0}^n g_k(q) t^k$ where $g_k(q)\in \mathbb{Z}[q]$. If the harmonic algebra of $Z$ is Gorenstein, the coefficients $g_k(q)$ and $g_{n-k}(q)$ are related by a quantum symmetry. This symmetry is evocative of Hibi's palindromicity theorem which states that reflexive polytopes have symmetric $h^*$ vectors \cite{hibi}.

\begin{prop}\label{prop:palin}
  Suppose that $Z$ has a Gorenstein harmonic algebra. If $M$ is Boolean, then for all $k$,
  \[g_k(q)= q^{\binom{n}{2}} g_{n-k-1}(q^{-1}).  \]
  If $M$ is not Boolean, then for all $k$,
  \[g_k(q)= (-1)^{n+d}q^{\binom{n+1}{2}-d} g_{n-k}(q^{-1}). \]
\end{prop}
\begin{proof}
  If $M$ is not Boolean, then the proof of \Cref{thrm:gor} tells us that $\mathcal{H}_Z[1,0]\simeq \widetilde{\mathcal{H}_Z}$. Thus by \Cref{thrm:ehrhart_series},
  \[E_Z(t,q) = t^{-1} \widetilde{E_Z}(t,q) = (-1)^{d+1} t^{-1}q^{-d}E_Z(t^{-1},q^{-1}).\]
  We can compute that
  \begin{align*}(-1)^{d+1} t^{-1}q^{-d}E_Z(t^{-1},q^{-1}) &= (-1)^{d+1} t^{-1}q^{-d}\frac{\sum_{k=0}^n g_k(q^{-1})t^{-k}}{\prod_{i=0}^n (1-t^{-1}q^{-i})}\\ &= \frac{\sum_{k=0}^n (-1)^{n+d}q^{\binom{n+1}{2}-d} g_{k}(q^{-1})t^{n-k}}{\prod_{i=0}^n (1-tq^{i})}\,. \end{align*}
  This yields the second claim. The first claim follows from an identical argument and the fact that $\mathcal{H}_Z[2,0] \simeq \widetilde{\mathcal{H}_Z}$ when $M$ is Boolean.
\end{proof}

\begin{remark}\label{rem:euler-mahonian}
    In fact, a complete description of the numerator of the graded Ehrhart series of an $n$-cube $C_n$ is given in \cite[Section 4]{BB13} and \cite[Theorem 1.1]{BMC18}. Although, as those authors note, one could also derive it from the works of Carlitz and MacMahon using the equality $\Hilb(\text{Orb}(mC_n \cap \Z^n);q)1= [m+1]_q^n$; see the references in \cite{BB13} for more details. The polynomial $N_{C_n}(t,q)$ is an Euler--Mahonian distribution and is defined as follows. For a permutation $\pi=\pi_1\pi_2\cdots\pi_n$  of $S_n$, let $\text{Des}(\pi)= \{j\in [n-1]: \pi_{j}> \pi_{j+1}\}$, $\text{des}(\pi)= \vert \text{Des}(\pi)\vert$ and $\text{maj}(\pi) = \sum_{j\in \text{Des}(\pi)} j$. Then
    \[N_{C_n}(t,q) = \sum_{\pi\in S_n} t^{\text{des}(\pi)}q^{\text{maj}(\pi)} .\]
    The first identity given in \Cref{prop:palin} is equivalent to the identity
    \[\sum_{\substack{\pi\in S_n\\\text{des}(\pi)=k}}q^{\text{maj}(\pi)} = \sum_{\substack{\pi\in S_n\\\text{des}(\pi)=n-k-1}}q^{\binom{n}{2}-\text{maj}(\pi)}.\]
\end{remark}

\begin{remark}
A similar argument to the proof of \Cref{thrm:gor} shows that the affine semigroup ring of a unimodular zonotope $Z$ is Gorenstein if and only if $M$ is Boolean or every connected component of $M$ is a circuit.
\end{remark}

\begin{remark}
\Cref{prop:palin} gives us partial information about the numerator of the graded Ehrhart series of Gorenstein unimodular zonotopes. As $i_Z(0;q)=1$, it can be checked that $g_0(q)=1$ for all unimodular zonotopes $Z$. If $M$ is Boolean, \Cref{prop:palin} implies $N_Z(t,q)$ has $t$-degree $n-1$ with $g_{n-1}(q)= q^{\binom{n}{2}}$. If the harmonic algebra of $Z$ is Gorenstein and $M$ is not Boolean, then \Cref{prop:palin} implies that $N_Z(t,q)$ has $t$-degree $n$ with $g_{n}(q)= (-1)^{n+d}q^{\binom{n+1}{2}-d}$.
\end{remark}

\section{Further questions}\label{sec:questions}
Here we collect some questions for future work.
\subsection{A matroidal formula for the graded $h^*$ vector}
Given a matroid $M$ of rank $d$ on a ground set of size $n$, define $h^*(t)\in \Z[t]$ and $h^*_q(t;q)\in \Z[t,q]$ by declaring
\begin{gather*}
\sum_{m\geq 0} m^d T_M(\frac{m+1}{m},1)t^m = \frac{h^*(t)}{(1-t)^{d+1}} \quad \text{and}\\ \sum_{m\geq 0} q^{(n-d)m}[m]_q^d T_M(\frac{[m+1]_q}{[m]_q},q^{-m})t^m = \frac{h^*_q(t;q)}{(1-t)(1-tq)\cdots (1-tq^n)}\,. \end{gather*}
\Cref{lem:formal_rationality} implies that $h^*_q(t;q)$ is well defined and that $h^*_q(t;1) = (1-t)^{n-d}h^*(t)$. 

If $M$ is the matroid of a unimodular zonotope $Z$, $h^*(t)$ and $h^*_q(t;q)$ are the numerators of the Ehrhart series and graded Ehrhart series of $Z$, respectively. In this case, \cite[Corollary 5.10]{BJM19} gives an expression for $h^*(t)$ as a sum of $j$-Eulerian polynomials. This sum is indexed by the bases of $M$ and relies on the activities of these bases as input. In light of \Cref{rem:euler-mahonian}, it seems possible that $h^*(t;q)$ might have a similar expression as a weighted sum of $q$-analogues of the $j$-Eulerian polynomials. 

\begin{problem}
    Give an explicit matroidal formula for $h^*_q(t;q)$.
\end{problem}

\subsection{The graded Ehrhart theory of non-unimodular zonotopes}
If $Z$ is an arbitrary zonotope, the Ehrhart theory of $Z$ is dictated by the \emph{arithmetic matroid} $\mathfrak{M}$ of $Z$. The arithmetic matroid $\mathfrak{M}$ associated to a zonotope is the data of a matroid $M$ along with a multiplicity function $m: 2^E\to \Z$ which has the property that for a basis $B$ of $M$, $m(B)$ is equal to the determinant of the submatrix indexed by $B$. Every arithmetic matroid has an arithmetic Tutte polynomial $T_{\mathfrak{M}}(x,y)$. In the case of a unimodular zonotope $Z$, the arithmetic Tutte polynomial of the arithmetic matroid of $Z$ is equal to ordinary Tutte polynomial of the matroid associated to $Z$. In \cite{DAM12}, it is shown that 
\[ \vert mZ \cap \Z^d \vert  =  m^d 
    T_{\mathfrak{M}}(\frac{m+1}{m},1)\,\,\,\text{and}\,\,\,
    \vert \text{int}(mZ) \cap \Z^d \vert = m^d 
    T_{\mathfrak{M}}(\frac{m-1}{m},1)\]
for all zonotopes $Z$. This expression generalizes the one given in \Cref{prop:ehrhart_poly} outside of the realm of unimodular zonotopes. One could ask whether our \Cref{thrm:ehrhart_series_intro} generalizes in a similar fashion. As the following example shows, this is not true, at least in its most obvious form.

\begin{example}
Let $Z\subseteq \R^2$ be the zonotopal diamond which is the image of $[0,1]^2$ under the linear map given by $A= \begin{bsmallmatrix}
    1 & 1\\
    -1 & 1
\end{bsmallmatrix}$. This zonotope is not unimodular as $\det(A) = 2$. The arithmetic Tutte polynomial associated to $Z$ is $T_{\mathfrak{M}}(x,y) = x^2+1$. Through direct calculation we see
\[i_Z(1;q)= 1+2q+2q^2 \neq 2+2q+q^2=T_{\mathfrak{M}}([2]_q,1). \]
\end{example}

Despite this example, it is still possible that the $q$-lattice point counts of an arbitrary zonotope are an arithmetic matroid invariant. It would be quite surprising to find two zonotopes $Z$ and $Z'$ whose arithmetic matroids are the same but $i_Z(m;q)\neq i_{Z'}(m;q)$ for some integer $m\geq 0$.
\begin{question}
    Is the graded Ehrhart theory of a non-unimodular zonotope $Z$ an invariant of its associated arithmetic matroid? If so, is there a generalization of \Cref{prop:ehrhart_poly} to arithmetic matroids?
\end{question}

\subsection{Homological aspects of $\mathcal{H}_Z$}
The harmonic algebras $\mathcal{H}_Z\simeq \mathbb{C}[z_S\vert S\subseteq [n]]/I_{L}^{\text{SE}}$ warrant further study. In particular, it would be interesting to give a bigraded free resolution of $\mathcal{H}_Z$ as a $\mathbb{C}[z_S\vert S\subseteq [n]]$ module where $\deg(z_S)= (1,|S|)$. 
Already for the case of a cube, this question is interesting. In this case, $\mathcal{H}_Z$ is the coordinate ring of the Segre embedding of $(\P^1)^n \subseteq \P^{2^n-1}$. 
In a seminal paper \cite{Snowden}, Snowden showed that there is a finite list of ``master syzygies'' which control all syzygies of Segre embeddings. 

Let $R=\mathbb{C}[x_0,x_1,\ldots,x_n]$ be a bigraded polynomial ring where $x_i$ has degree $(1,i)$. Consider the map $R\to \mathcal{H}_Z$ defined by sending $x_i\mapsto \sum_{|S|=i} z_S$. Braun and Olsen prove that if $Z$ is a cube, then $\mathcal{H}_Z$ is a finitely generated and free $R$-module \cite[Proof of Theorem 4.10]{BMC18}. As every harmonic algebra of a unimodular zonotope is a quotient of the harmonic algebra of a cube, this implies that every harmonic algebra is a finitely generated $R$-module. Hilbert's syzygy theorem ensures that $\mathcal{H}_Z$ has a finite free resolution as an $R$-module. In fact, as $\mathcal{H}_Z$ is Cohen--Macaulay (\Cref{thrm:CM-intro}), the Auslander--Buchsbaum formula (\cite[Chapter 19]{Eisenbud}) ensures that this free resolution will have length $n+1-(d+1)=n-d$. By describing the bigraded Betti numbers of such a resolution, one will obtain a description of the numerator of the graded Ehrhart series of $Z$ in \Cref{thrm:ehrhart_series}. Our partial descriptions of the numerator of the graded Ehrhart series gives some information on the shape of these Betti numbers. See \Cref{thrm:ehrhart_series} and \Cref{prop:palin}.

\begin{problem}\label{prob:hom}
    Give a free resolution of $\mathcal{H}_Z$ as a $\C[z_S\vert S\subseteq [n]]$-module or as an $R$-module. 
\end{problem}

As we've given a presentation for the homogeneous coordinate ring of any complex arrangement Schubert variety $Y_L\subseteq \P^{2^n-1}$ under the Segre embedding, one could ask the same questions as \Cref{prob:hom} for these coordinate rings. By \Cref{thrm:geometry}, such coordinate rings are isomorphic to a harmonic algebra $\mathcal{H}_Z$ when $L$ is unimodular.

\printbibliography
\end{document}